\documentclass[12pt,leqno]{article}
\usepackage{amsmath}\usepackage{amsthm}
\usepackage{amssymb}
\usepackage{graphicx}
\usepackage{epsfig}     
\usepackage{color}
\numberwithin{equation}{section}
\newtheorem{theorem}{Theorem}[section]
\newtheorem{proposition}{Proposition}[section]
\newtheorem{lemma}{Lemma}[section]
\newtheorem{corollary}{Corollary}[section]
\newtheorem{remark}{Remark}[section]

\newtheorem{definition}{Definition}[section]
\renewcommand{\dfrac}{\displaystyle\frac}
\newcommand{\dint}{\displaystyle \int}

\newcommand{\ren}{\mathbb{R}^n}

\newcommand{\RR}{\mathbb{R}}
\newcommand{\mr}{\mathbb{R}}

\newcommand{\ve}{\varepsilon}

\newcommand{\p}{\partial}
\newcommand{\ep}{\varepsilon}
\newcommand{\nc}{\normalcolor}
\newcommand{\nab}{\nabla}
\newcommand{\hal}{\frac{1}{2}}

\def\({\left(}
\def\){\right)}

\def\a{\alpha}
\definecolor{darkblue}{rgb}{0.05, .05, .65}
\definecolor{darkgreen}{rgb}{0.1, .65, .1}
\definecolor{darkred}{rgb}{0.8,0,0}

\begin{document}

\title{ A Mean Field Equation as  Limit of Nonlinear  \\ Diffusions with Fractional Laplacian Operators}

\author{  Sylvia Serfaty and Juan Luis V\'azquez}

\date{June 2012}

\maketitle
\pagestyle{myheadings} \thispagestyle{plain} \markboth{}{}
\maketitle
\begin{abstract}

 In the limit of a nonlinear  diffusion model involving the fractional Laplacian we
get a ``mean field"  equation   arising in  superconductivity and superfluidity.
For this equation, we obtain uniqueness, universal bounds and regularity results. 
We also show that  solutions with finite second moment and radial solutions admit an asymptotic large time limiting profile which is a  special self-similar solution: the ``elementary vortex patch". \end{abstract}

\vskip .5cm


{\bf Keywords and phases.} Nonlinear diffusion, fractional Laplacian, hydrodynamic limit, superconductivity, vortex patch, universal bound.


\section{Introduction}

We consider the class of evolution models with nonlocal effects  given by the system
\begin{equation}\label{eq1}
\sl  u_t=\nabla\cdot(u\,\nabla p), \quad p={\cal K}u.
\normalcolor
\end{equation}
where $u$    is a function of the variables $(x,t)$ to be thought of
as a density or concentration, and therefore nonnegative, while
$p$ is the ``pressure" and ${\bf v}=-\nabla p$ represent the mass propagation speed. We postulate that $p$ is related to $u$ via a linear integral operator $\cal K$  which we assume in practice to
be the inverse of a fractional Laplacian,  that is,
 ${\cal K}=(-\Delta)^{-s}$, with $s\in (0,1]$. More general integral operators are also now being considered in the literature.

The problem is posed for $t>0$ and $x\in \RR^n$,  $n\ge 2$. (The case $n=1$ is easier  and a bit different,  and a  short explanation will be given below). Finally, we
give initial conditions
\begin{equation}\label{eq.ic}
u(x,0)=u_0(x), \quad x\in \RR^n,
\end{equation}
where $u_0$ is a nonnegative  integrable function in
$\RR^n.$

\medskip

\noindent {\sc Nonlocal diffusive models.} These correspond to the equation
\begin{equation}\label{eq.s}
  u_t=\nabla\cdot(u\,\nabla (-\Delta)^{-s}u),
\normalcolor
\end{equation}
in the  parameter range $0<s< 1$. This equation can  be considered as a nonlocal version of the standard porous medium equation $u_t=\Delta (u^2)$.
 For convenience, we will often refer to it as the FPME (fractional porous medium equation).

The FPME has been derived starting from the continuity equation
\begin{equation}
\partial_t u+ \nabla\cdot(u{\bf  v})=0\,,
\end{equation}
where $u$ is a density and ${\bf v}$ is a velocity field, which
according to Darcy's law derives from a pressure, ${\bf v}=-\nabla
p$ (as in the theory of gases in porous media).  Physical applications range from
macroscopic evolution of particle systems with short- and long-range
interactions \cite{GLM2000}, to phase segregation \cite{GL1, GL2}, and dislocation dynamics \cite{Head, BKM}.
 The existence of a weak solution, the
relevant integral estimates, as well as the property of compact
support have been established by Caffarelli and the second author
in the paper \cite{CV1}.  The existence is shown by passing to the
limit  in the sequence of smooth solutions of approximate
problems, and using the compactness of the integral operators
involved in the weak formulation and its energy estimates for
$s<1$ (thus the method does not directly apply to $s=1$). Note that no uniqueness is proven,
and examples are constructed where the maximum principle does not hold for $s>1/2$.

\medskip

\noindent {\sc Hydrodynamic or mean field equation.} The case $s=1$ of equation \eqref{eq1}--\eqref{eq.ic}
\begin{equation}\label{eqs1}
 u_t
= \nabla\cdot(u\,\nabla p),\qquad p=(-\Delta)^{-1}u\,,
\end{equation}
supplemented with initial data:
\begin{equation}\label{init}
u(x,0) = u^0(x)\,,
\end{equation}
is no more a diffusion equation. This equation was first studied in  $\RR^2$ by Lin-Zhang \cite{lz}.   There, existence is proven by a vortex point approximation, and uniqueness by estimates in Zygmund spaces. In dimension 2,  the equation is directly related to the
Chapman-Rubinstein-Schatzman mean field model  of
superconductivity \cite{crs} and to  E's model of superfluidity \cite{E}, which would correspond rather to the equation
\begin{equation}\label{signeq}
u_t= \nab \cdot (| u|\nab p)\end{equation}  which coincides with \eqref{eqs1} when $u\ge 0$.
 That model is a
mean field model for the motion of vortices in a superconductor in
the Ginzburg-Landau theory. There, $u$   represents the local
vortex-density, and $p$  represents the induced magnetic field in
the sample. In this application the equation is really posed in a
bounded domain and for signed ``vorticity measures" $u$ (which can
be measures instead of functions), but the study of \eqref{eqs1}
in $\RR^2$ with $u\ge 0$, corresponding to the situation where the vorticity is positive, is a reasonable start.  It can also be
viewed as the gradient-flow version of the Euler equation in
vorticity form i.e.
$$
u_t= \nabla \cdot (u \, \nab^\perp p) \qquad p=(-\Delta)^{-1} u\,,
$$
hence also  the term hydrodynamic equation. In \cite{AmSr}, Ambrosio and the first author studied the
model in a bounded  two-dimensional  domain (taking into account the possibility of
flux of vortices through the boundary), via a gradient flow
approach: the equation is the gradient-flow for a natural energy
(see below) for the $2$-Wasserstein distance. There, $u$ is
denoted $\mu$, and $p$ is denoted $h$. Solutions with $u$ measure
(and not $L^1$) are also treated there. The equation \eqref{signeq} with changing sign $u$  has been studied by Masmoudi-Zhang \cite{mz} in $\RR^2$,
Ambrosio-Mainini-Serfaty \cite{ams} in bounded  domains of the plane and
in the whole plane, again using a gradient-flow approach.

This equation is also related to nonlocal aggregation models, which have attracted a lot of attention lately, notably in conjunction with the Keller-Segel  model \cite{KS71}, see e.g. \cite{bcl,bgl,BLL}. They correspond to \eqref{eq1} with $\cal K$ a general integral operator corresponding to an attractive (instead of repulsive) interaction. Particularly relevant to us is the paper of Bertozzi-Laurent-L\'eger \cite{BLL} which deals mostly with aggregation via the Coulomb kernel, but contains a section (Section 3)  that deals with the repulsive case, i.e. exactly the equation \eqref{eqs1}-\eqref{init}. In addition to existence and uniqueness results, it makes  a detailed analysis of the property of support propagation and asymptotic behaviour for solutions with compactly supported initial data.

\medskip

\noindent {\sc Outline of results.} In the present paper we are interested in obtaining Equation \eqref{eqs1}  as the limit of the nonlocal diffusive equations \eqref{eq.s} when $s\to 1$;  and in  expanding the theory for  the limiting equation.  We consider the initial-value problem  posed  in $\RR^n$, for all dimensions $n\ge 2$, with comments on the easier case $n=1$. We develop an existence theory for weak solutions with a  nonnegative finite measure as initial data.  Such solutions are uniformly bounded for positive times in the form $|u(x,t)|\le 1/t$. We prove a  uniqueness result for bounded solutions.

As another important feature of the theory, we derive the support propagation properties (for general solutions), and compare the result with a similar property for the fractional diffusion approximation.

Some insight is obtained with the introduction of the special formulation  of the equation under conditions of radial symmetry, leading to a non-homogeneous Burgers' equation for the so-called mass function. With this Burgers' equation approach, we recover the existence of a fundamental solution in the form of a round vortex patch of constant (in space) density,  and show that they are attractors for the evolution of all radial solutions.  We also find  other explicit radial  solutions that  allow to derive a new non-comparison result.\nc

 We also study the question of large time behaviour for another  class of initial data where the restriction of compact support is dropped, more precisely  integrable data with finite second moment. The method, which is interesting for its own sake, is an energy method that follows the now well-known entropy - entropy dissipation approach,  see \cite{CV2} and  references.

\medskip

\noindent {\sc Plan of the paper.}
The paper is organized as follows: Section \ref{sec.exist} is devoted to the definition of solutions for $s<1$ and $s=1$ and to recalling some of their properties (conserved quantities, etc). In Section \ref{sec2} we show the existence of solutions to \eqref{eqs1} by taking the limit  $ s \to 1$ in solutions to \eqref{eq1}. Section \ref{secuniq} is devoted to the proof of uniqueness for bounded solutions. Section \ref{sec.ub} contains the universal decay estimate for the solutions. In Section \ref{sec.radial} we investigate the theory in the class of radial solutions via the Burgers' equation approach. In Section \ref{sec.compact} we show that the compact support property is preserved during the evolution, see also \cite{BLL}.  Section \ref{sec.asbeh} discusses the question of large time behaviour for two classes of initial data where the restriction of compact support is dropped: (i) integrable data with finite second moment, (ii) all solutions that are radially symmetric in the space variable.

\medskip
\noindent {\sc Open problems.}
Here is list of the main problems that remain open: obtaining uniqueness results beyond the class of bounded solutions, analyzing the long-term behaviour of all  solutions with integrable data, and finally treating the case of sign changing solutions.

\medskip
\noindent {\sc Acknowledgments.} We would like to thank J. A. Carrillo for pointing out to us the reference \cite{BLL} after a first draft of this paper was written. S. S. thanks the Universidad Aut\'onoma de Madrid for its hospitality  that allowed this work to be completed. S. S. was supported by a EURYI award, and JLV   by
Spanish Project MTM2008-06326-C02-01.
\nc

\section{Definitions and first properties of the solutions}\label{sec.exist}

 It will be
convenient to write for any $s\in (0,1]$ the equation in the  form
\begin{equation}\label{gen.eq}
\partial
_t u=\nabla\cdot (u\,\nabla {\cal K}_s u),\qquad {\cal K}_s=(-\Delta)^{-s}\,.
\end{equation}
 ${\cal K}_s$ is a positive essentially self-adjoint
operator. For convenience, we introduce  ${\cal H}_s={\cal
K}_s^{1/2}=(-\Delta)^{-s/2}$. We write $p={\cal K}_s u$, so that
$u=(-\Delta)^s p$ and ${\cal H}_s u=(-\Delta)^{s/2} p$.

We start with the definition of weak solutions.
 The notation $\mathcal{M}^+(\mr^n)$ denotes the space of positive  Radon  measures  on $\mr^n$ with finite mass.

\begin{definition}\label{defsolweak}For any $s\in (0,1]$, a  weak solution of  \eqref{gen.eq} in $Q_T=\RR^n\times (0,T)$ is a
nonnegative function $u(x,t)$ such that $u\in L^1(Q_T)$, $p={\cal
K}_s u \in L^1((\tau,  T),  W^{1,1}_{loc}(\RR^n))$,
$u\,\nabla p\in L^1((\tau,  T)\times \ren)$ for all $\tau>0$, and the identity
\begin{equation}
\iint (u\,\phi_t- u\, \nabla p\cdot\nabla\phi)\,dxdt= 0
\end{equation}
holds for all continuously differentiable  test functions  $\phi$ compactly supported  in $Q=\mr^n\times (0, T)$. If $s<1$ we also assume that $u$ is continuous in $Q_T$. \end{definition}

\begin{definition}\label{defsol}For any $s\in (0,1]$, a  weak solution of the problem formed by equation   \eqref{gen.eq}  in $Q_T=\RR^n\times (0,T)$ with initial data $\mu \in {\cal M}^+(\RR^n)$ is a continuous
nonnegative function $u(x,t)$ with  $u\in L^1(Q_T)$, $p={\cal
K}_su \in L^1((0,  T),  W^{1,1}_{loc}(\RR^n))$,
$u\,\nabla p\in L^1(Q_T)$, and the identity
\begin{equation}
\iint (u\,\phi_t- u\, \nabla p\cdot\nabla\phi)\,dxdt+ \int
\phi(x,0)\,d\mu(x)=0
\end{equation}
holds for all continuously differentiable  test functions $\phi$ in $\mr^n\times [0, \infty)$  such that
$\phi$ has compact support in
the space variable and vanishes for $t\ge T$. \end{definition}

In practice the solutions obtained below have much better properties that can be exploited in the analysis. In particular they will be uniformly bounded for $t\ge \tau>0$ and satisfy energy inequalities. Following \cite{lz} we may call such improved solutions {\sl dissipative solutions}.

\subsection{Properties of the solutions in the case $s<1$}

The existence of weak solutions in the sense of Definition \ref{defsol} is proved in \cite{CV1} under the assumption that the initial data is actually a bounded function $u_0\ge 0 $ decaying exponentially as $|x|\to\infty$, cf. Theorem  4.1 of \cite{CV1}.  It is extended to general nonnegative $L^1$ initial data in \cite{CSV}.
Let us list the main properties of those solutions with attention to the way they depend on $s\in (0,1)$. We point out that we do not use all of them in the sequel.

- {\sc Conservation of mass:} The solutions exist globally in time and for every $t>0$ we have
\begin{equation}
\int u(x,t)\,dx=\int u_0 (x)\,dx.
\end{equation}

- {\sc Conservation of positivity:} $u_0\ge 0$ implies that $u( \cdot, t)\ge 0$
for all times.

 -{\sc  $L^p $ estimates:} the $L^p$ norm of the constructed solutions does not increase in time, for any $1\le p\le \infty$.

- A general {\sc comparison theorem,} i.\,e., a form of the usual
maximum principle,  is proven to be false by constructing counterexamples for $s>1/2$, see \cite{CV1}. However, comparison of solutions for an integrated version of the equation
is established in dimension $n=1$ by the method of viscosity
solutions in Biler-Karch-Monneau \cite{BKM}. We will return to that issue below.

- {\sc Finite speed of propagation.} If the initial function has compact support so does the solution at any given time $t>0$. Estimates on the growth of the support for very large time
are given in \cite{CV1} and \cite{CV2}.

 -{\sc  Log-entropy estimate:} We have
\begin{equation}
\frac{d}{dt}\int u(x,t)\log u(x,t)\,dx=-\int |\nabla {\cal
H}_s u|^2\,dx.
\end{equation}
therefore, solutions with $u_0\log  u_0\in L^1(\ren)$ stay in the
same space and moreover satisfy
$$
\iint |\nabla {\cal H}_{s} u|^2\,dxdt<   \infty.
$$

- {\sc  Energy estimate.} We first make the observation that
if $p={\cal K}_s u$ we have
\begin{equation}\label{energie}
E(u):=\int up\,dx=\int ((-\Delta)^{s/2}p)^2\,dx =\int ({\cal
H}_s u)^2\,dx \ge 0.
\end{equation}
This formal computation is justified for the class of constructed
weak solutions. This is the main energy we are going to use. We
also have formally
$$
\int |\nabla {\cal H}_s u|^2\,dx=\int|\nabla
(-\Delta)^{s/2}p|^2\,dx=\int u(-\Delta p)\,dx.
$$
The estimate on the evolution of the energy according to our flow
is then
\begin{equation*}
\frac{d}{dt}\int |{\cal H}_s u|^2\,dx=2\int ({\cal H}_s u ) ({\cal H}_s u )_t\,dx= 2\int ({\cal K}_s u ) u_t\,dx= -2\int u|\nabla {\cal
K}_s u |^2\,dx.
\end{equation*}
where the last identity is obtained after using the equation and integrating by parts. This is justified for smooth solutions, while in the generality of weak solutions we get
\begin{equation}\label{2.12}
\int_{\ren} |{\cal H}_s u(t_2)|^2\,dx + 2\int_{t_1}^{t_2} \int u|\nabla {\cal K}_s u|^2\,dx\le \int_{\ren} |{\cal H}_s u(t_1)|^2\,dx
\end{equation}

- {\sc Boundedness: } The solutions to \eqref{gen.eq} constructed in \cite{CV1} have the following {\sl boundedness effect}  (used in the subsequent paper \cite{CSV}): For all $0<s<1$, weak energy solutions in $L^1$ are bounded for all
positive times.  More precisely, there is a constant $C=C(n,s)$ such that
\begin{equation}\label{2.9}
\|u(t)\|_\infty\le C(n,s) \, t^{-\alpha}\|u_0\|_{L^1}^{\sigma}
\end{equation}
Dimensional considerations imply that $\alpha=n/(n+2-2s)$ and $\sigma
= (2-2s)/(n+2-2s)$. This fact is proved in \cite{CSV}. Notice that as $s\to1$ we have $\alpha\to 1$ and $\sigma\to 0$. Once we can control the constant $C(n,s)$ uniformly in $s$ we can pass to the limit $s\to 1$ to obtain the a priori estimate (already found  by Lin and Zhang in their study in dimension $n=2$ in \cite{lz})
\begin{equation}
u(x,t)\le C/t
\end{equation}
for the solutions of the limit equation obtained in the limit process, this will be done in the next section in Theorem \ref{thm.exlimit}.  This just requires  a closer inspection of the proof of \cite{CSV}.

- {\sc Energy solutions: } The basis of the boundedness analysis is a property that
goes beyond the definition of weak solution. We will review the formulas with attention to the constants that appear since this is not done in \cite{CSV}. The general energy property is as follows: for any $F$ smooth and such that $f=F'$ is bounded and nonnegative, we have   for every $0\le t_1\le t_2\le T$,
\begin{equation*}\label{entro}
\begin{array}{ll}
\dint F(u(t_2))\, dx -\dint F(u(t_1))\, dx & = - \dint_{t_1}^{t_2}\dint \nab [f(u)] u \nab p\, dx\, dt= \\
& -\dint_{t_1}^{t_2}\int \nab h(u) \nab (-\Delta)^{-s} u\, dx\, dt
\end{array}
\end{equation*}
where  $h$ is a function  satisfying  $h'(u)= u\,f'(u)$. We can write the last integral as a bilinear form
\begin{equation}
\int \nab h(u) \nab (-\Delta)^{-s} u\, dx = \mathcal{B}_s (h(u), u)
\end{equation}
and this bilinear form $\mathcal{B}_s$ is defined on the  Sobolev space $W^{1,2}(\ren)$ by
\begin{equation}\label{B}
\mathcal{B}_s(v,w)= C_{n,s} \iint  \nab v(x)\frac{1}{|x-y|^{n-2s}} \nab w(y)\, dx \, dy = \iint {\cal N}_{-s} (x,y) \nab v(x) \nab w(y) \, dx\, dy\end{equation}
where ${\cal N}_{-s}(x,y)= C_{n,s}|x-y|^{-(n-2s)}$ is the kernel of operator $(-\Delta)^{-s}$.
After some integrations by parts we also have
\begin{equation}\label{BB}
\mathcal{B}_s (v,w)= C_{n,1-s} \iint(v(x)-v(y)) \frac{1}{|x-y|^{n+ 2(1-s)}} (w(x)-w(y))\, dx\, dy\end{equation}
since $-\Delta {\cal N}_{-s}={\cal N}_{1-s}$. It is known \cite{Stein} that $\mathcal{B}_s (u,u)$ is an equivalent norm for the fractional Sobolev space $W^{1-s,2}(\ren)$ and we will see just below  that  $C_{n,1-s}\sim K_n(1-s)$ as $s \to 1$, for some constant $K_n$ depending only on $n$.

\begin{proposition} \label{pro2.1}The constant $C(n,s)$ in \eqref{2.9} is uniformly bounded as $s\to 1$, i.e.  for any solution to \eqref{gen.eq}, we have
$$\|u(t)\|_{\infty} \le \frac{C_n \|u_0\|_{L^1}^\sigma  }{t^\alpha} ,$$
 with
 $\alpha=n/(n+2- 2s)$ and $\sigma = (2-2s)/(n+2-2s)$.

 \end{proposition}

\begin{proof} It is based on a careful scrutiny of the ideas and proof of paper \cite{CSV}, Section 4.\\
\noindent (i) We  have   $C_{n,1-s}\sim K_n\,(1-s)$ as $s \to 1$. Indeed, if $0<s<1$ we can also use the  integral representation
\begin{equation}
(-\Delta)^{s}  g(x)= C_{n,s }\mbox{
P.V.}\int_{\mathbb{R}^n} \frac{g(x)-g(z)}{|x-z|^{n+2s}}\,dz  = {\cal N}_{s}\ast g (x)  \label{def-riesz}
\end{equation}
where P.V. stands for principal value and $C_{n,s}=\frac{4^{s}s\Gamma((n/2)+s)}{\pi^{n/2}\Gamma(1-s)}$
is a normalization constant, see for example \cite{Landkof,Stein}. Note that $C_{n,s }\sim K_n s$ as
$s\to 0$ and $C_{n,s }\sim C_n'(1-s)$ as $s\to 1$.

\medskip

\noindent (ii)
Next, we need the following result from  \cite{MazSh02}
$$
\lim_{r\to 0^+} r\,\iint \frac{(u(x)-u(y))^2}{|x-y|^{n+2r}}dxdy = |\mathbb{S}^{n-1}|\int u(x)^2\,dx.
$$
Setting $r=1-s$ and combining this with \eqref{BB} and $C_{n,1-s} \sim K_n (1-s)$ as $ s\to 1^-$, we deduce  that
$$
\lim_{s\to 1^-}  {\cal B}_s(u,u)= K_n|\mathbb{S}^{n-1}|\|u\|_2^2.
$$

\noindent (iii) We can now revisit the proof of boundedness of Theorem 4.1, in Section 4 of \cite{CSV}, and check that the constants in all the arguments are uniform in $s$ as $s\to 1$.
\end{proof}

We end this section with some  further properties of the solutions.

\medskip

- {\sc H\"older estimates} are proven for bounded energy solutions in  \cite{CSV} and hold for all $s\in (0,1)$. The H\"older exponent has to deteriorate as $s\to 1$ since we do not expect solutions of the limit equation to be continuous,  as can be seen in the case of the  ``vortex-patch solutions" below.

- {\sc More general data.} The property of boundedness and the $C^\alpha$ regularity are used in \cite{CSV} to obtain solutions with similar properties for all nonnegative initial data in $L^1(\ren)$. We will generalize that result to measures in Section \ref{sec2} below.


\section{Existence  of solutions  as limits of fractional diffusion}\label{sec2}

In dimension $n=2$,  existence and uniqueness of positive  $L^\infty$ solutions has been proven by Lin and Zhang in \cite{lz} (they also proved existence of positive measure-valued solutions). Existence with positive  initial data of finite energy has also been  proven by a gradient flow approach, in bounded domains of the plane in \cite{AmSr},  and in all of $\RR^2$ in \cite{ams}. In general dimension, existence can  be obtained by taking the limit $s\to 1$ in the solutions for $s<1$ constructed in \cite{CV1} as we will see below.

We examine here the situation $s=1$ as limit of the equations for $s<1$.
Since all the estimates above are uniform in $s$ as $s\to 1$,  we may
write the formal version for $s=1$, where the equation is
\eqref{eqs1}.
 Here we need to distinguish the case $n\ge 3$ from the case $n=2$ which we will examine separately below.
For $n \ge 3$, and noting that now ${\cal H}u =(-\Delta)^{1/2}p$,  the energy is
\begin{equation}\label{energys1}
E(u)=\int_{\RR^n} up\,dx=\int ({\cal H}u)^2\,dx= \int_{\RR^n} |\nabla p|^2\,dx,
\end{equation}
 which is an analogue of the energy of \cite{AmSr} (there, it was set in bounded domains) for which the equation is a Wasserstein gradient flow,
and its evolution is given by
\begin{equation}
\frac{dE}{dt}=-2\int_{\RR^n} u|\nabla p|^2\,dx.
\end{equation}
 This is a formal computation that will be justified as a consequence of the passage to the limit $s\to 1$ in formula \eqref{2.12}.
This energy would not make sense for $n=2$ because, unless $\int
u=0$, $p=(-\Delta)^{-1} u$ is never of gradient in $L^2$ in
dimension $2$ (the Green kernel is a logarithm), however a definition is given below in Section \ref{n2}.

\subsection{The existence result by $ s \to 1$ limit}
We next prove the following general result in all dimensions $n\ge 2$.

\begin{theorem}\label{thm.exlimit} For every $\mu \in {\cal M}^+ (\ren)$, there exists a continuous and nonnegative weak solution of equation \eqref{eqs1},  with data $\mu $
 in the  sense of Definition \ref{defsol}. It can be obtained as the limit $s\to 1$ of the solutions of the FPME with $0<s<1$ with data $u_k^0\in L^1(\ren)\cap L^\infty(\ren)$ that approximate $\mu$ in the sense of measures. In addition these solutions satisfy
\begin{equation}
\label{linfbound}
\|u(t)\|_{L^\infty} \le \frac{C}{t}\end{equation} for some $C>0$ depending only on $n$.
\end{theorem}

\begin{proof}
Existence is proved by passing to the limit $s\to 1$ in the solutions of the FPME \eqref{eq1}. The outline is similar to the proof of Theorem 13.1 in \cite{CSV} where $s$ was considered constant, but we cannot use the continuity  in time of the solution, since it may be lost in the limit $ s \to 1$.

 Let us consider initial data $u^0=\mu $ in $\mathcal{M}^+(\mr^n)$,  and a sequence $u_k^0$  of nonnegative functions approaching $u^0$ as measures as $ k \to \infty$. (If $u^0\in L^1$ we may simply take $u^0_k=u^0$.) Let us then    consider the family of solutions
$u_k(x,t)$   with parameter $s_k\to 1$  of \eqref{eq1} given by Theorem 4.1  of \cite{CV1} or Theorem  13.1 of \cite{CSV}, and let $p_k(x,t)=(-\Delta)^{-s} u_k(x,t)$ be the corresponding pressures.

\noindent (i) In a first step, we assume that $n\ge 3$, and prove that the limit exists (along subsequences $s_{k_j}\to 1$) and is a  solution of the equation in the sense of Definition~\ref{defsolweak}, without discussing the question of  the initial data.
 By Proposition \ref{pro2.1}, we have that for any $\tau>0$,
\begin{equation}\label{unifuk}
\|u_k\|_{L^\infty ((0, \tau), L^1_x)} \le C\end{equation}
where $C>0$ only depends on $\tau $.
  In order to  pass to the limit in the weak formulation
\begin{equation}
\iint (u_k\,\phi_t- u_k \nabla p_k \cdot\nabla\phi)\,dxdt=0
\end{equation} for $\phi$ compactly supported in $\mr^n \times (0,T)$,  it suffices to have  weak convergence of $u_k$, together with  strong convergence of $p_k$ and $\nabla p_k$.
Let us work on $[\tau, T)$  containing the support of $\phi$.
From \eqref{unifuk} we have weak convergence (up to subsequence extraction) of $u_k$ in $L^q([\tau, T)\times B_R)$ for any $q\ge 1$, $R>0$.

- For $p_k$, we may invoke  the embedding theorems associated to the Riesz kernels, which tell us that $p_k=(-\Delta)^{-s} u_k $ is uniformly bounded in $L^\infty((\tau, T), L^q_x(B_R))$ for all large $q$ and $R>0$. More precisely, $q>n/(n-2s)$, which is uniform in $s\approx1 $ if $n\ge 3$.

- Moreover, $\nabla p_k$ has a similar bound in spaces of the form $L^\infty((\tau, T), W^{s',q}_{loc})$ with $s'<2s-1$ and convenient $q>n/(n-s)$. Since we are interested in $s\to 1$, we can restrict to  $2s-1>0$ so we can take $s'>0$ and both $(p_k)$ and $(\nabla p_k)$ are a locally compact family with respect to the space variable.

- We still need compactness in time  for $p_k$ and $\nabla p_k$, and this is obtained from Aubin-Simon's results,  cf. \cite{Aubin}, \cite{Simon}.
First, for $p_k$, we use the facts that
\begin{equation}\label{as1}
p_k \ \text{is bounded in } \ L^\infty((\tau, T), L^p( B_R))\end{equation}
\begin{equation*}\partial_t p_k = (-\Delta)^{-s_k}\partial_t u_k= (-\Delta)^{-s_k} \nab \cdot (u_k p_k)\end{equation*}
and since $u_k p_k$ is bounded in $L^q(B_R)$ for any $q>0$ and uniformly in $t\in [\tau, +\infty)$, it follows that
\begin{equation}\label{as2}\partial_t p_k  \ \text{is bounded in } \ L^\infty((\tau, T), W^{2s_k-1, q}(B_R)).\end{equation}
 \eqref{as1} and \eqref{as2} ensure, via Aubin-Simon's theorem, that $p_k$ is compact in $C([\tau, T], L^q(B_R))$, for any $q\ge 1$ and $R>0$.
The proof for $\nab p_k$ is entirely similar:
first, from the above, we have
\begin{equation}\label{as3}
\nab p_k  \ \text{is bounded in } \ L^\infty((\tau, T), W^{s'_k,q}(B_R))\end{equation}
\begin{equation*}\partial_t \nab p_k = (-\Delta)^{-s_k}\partial_t u_k= \nab (-\Delta)^{-s_k} \nab \cdot (u_k p_k)\end{equation*}
and since $u_k p_k$ is bounded in $L^q(B_R)$ for any $q>0$ and uniformly in $t\in [\tau, T]$, it follows that
\begin{equation*}\partial_t \nab p_k\  \text{is bounded in } \ L^\infty((\tau, T), W^{2s_k-2, q}(B_R)).\end{equation*}
Again, these suffice to apply Aubin-Simon's result and obtain that  $\nab p_k$ is compact in say $C([\tau, T], L^q(B_R))$ for every $\tau>0, T>\tau, R>0$, $q>1$. The limits after extraction of $u_k$ and $p_k$ are denoted $u$ and $p$.

Taking $\phi$ smooth compactly supported in space and vanishing for $t\ge T$, it follows from these convergence results  that
\begin{equation}\label{convttau}\lim_{k\to \infty}\int_{0}^\infty \int (u_k\phi_t - u_k \nab p_k\cdot \nab \phi)\, dx dt = \int_{0}^\infty \int (u   \phi_t - u \nab p \cdot \nab \phi)\, dx dt .\end{equation}

\medskip

(ii) When $n=2$ we do not have a uniform bound of the pressures $p_k$ in $L^q$ spaces valid for all $s\approx 1$, since this is a limit case of the Sobolev embeddings. Fortunately, the pressures do not appear directly in the equation, but only through the pressure gradients for which the estimates are uniform in the range $s\approx 1$. The rest of the argument follows.

\medskip

(iii) There remains to control terms corresponding to the  time interval $[0, \tau]$ in order to complete the proof that the solution takes on the initial data, as in Definition \ref{defsol}.  For that we follow \cite{CSV}, proof of  Theorem 13.1.  We consider $\phi\in C^1 (\mr^n\times [0, \infty)) $ compactly supported in space and vanishing for $t \ge T$. First, let $\tau>0$ be arbitrary. Applying the result of  \eqref{convttau} we have that
\begin{equation}\label{convttau2}\lim_{k\to \infty}\int_{\tau}^\infty \int (u_k\phi_t - u_k \nab p_k\cdot \nab \phi)\, dx dt = \int_{\tau}^\infty \int (u   \phi_t - u \nab p \cdot \nab \phi)\, dx dt .\end{equation}
In addition, by weak convergence of $u_k^0$ to $u^0$, we have
\begin{equation}\label{convttau3}\lim_{k\to \infty} \int u_k^0(x) \phi(x,0)\, dx= \int \int u^0(x) \phi(x,0)\, dx  .\end{equation} There remains to show that the contributions on the interval $[0, \tau]$ can be made small.

Starting again from Proposition \ref{pro2.1} we have
$$\|u_k(t)\|_{L^\infty}\le C(u_0) t^{-\alpha} \quad \alpha= n/(n+2-2s_k),$$
 which yields by conservation of mass
 $$\|u_k(t)\|_{L^p}^p \le \|u_k(t)\|_{L^1} \|u_k(t)\|_{L^\infty}^{p-1} \le C (u_0) t^{-\alpha(p-1)}.$$

By Riesz embedding we deduce
$$\int |\nab \mathcal{H}_{s_k} u_k|^2\, dx \le C \|u_k(t)\|_{L^p}^2 \le C t^{-\alpha(p-1)\frac{2}{p}},\qquad   \frac{1}{2}=\frac{1}{p}-\frac{s_k}{n}$$ where we recall $\mathcal{H}_s u=(-\Delta)^{-s/2}u$ for any $u$,
and combining the definitions of $\alpha$ and $p$, this yields
$$\int |\nab \mathcal{H}_{s_k} u_k|^2 \le C t^{\frac{2}{n+2-2s_k}-1}.$$
Combining with the energy inequality \eqref{2.12}, we deduce that we also have
\begin{equation}\label{decayup}
\int_{t}^\infty \int u_k |\nab p_k|^2 \le C t^{\frac{2}{n+2-2s_k}-1}.\end{equation}
We may now bound the remaining contribution
$\int_{0 }^\tau \int (u_k\phi_t - u_k \nab p_k\cdot \nab \phi)\, dx dt$
as  follows: for $t_m=2^{-m}$, we
first write $$\left|\int_{t_m}^{t_{m-1}} \int u_k \phi_t \right|\le t_{m-1}\|\phi\|_{C^1} \|u_k\|_{L^\infty((0, \infty), L^1)}\le C t_{m-1} $$
and second
\begin{multline*}\left|\int_{t_m}^{t_{m-1}} \int u_k \nab p_k\cdot \nab \phi\right|\le \| \phi\|_{C^1}
\(\int_{t_m}^{t_{m-1}} \int u_k \)^\hal\( \int_{t_m}^{t_{m-1}} \int u_k |\nab p_k|^2 \)^\hal\\
\le C t_{m-1}^\hal t_m^{\frac{1}{n+2-2s_k}-\frac{1}{2}   }
\le 2C t_m^\hal t_m^{-\hal + \frac{1}{n+2-2s_k}}= C t_m^{\frac{1}{n+2-2s_k}}.\end{multline*}
Summing over $m$ we find  that  the contributions on $[0 , \tau]$ tend to $0$ as $\tau $ tends to $0$ and combining with \eqref{convttau2}--\eqref{convttau3} we obtain
\begin{multline*}\lim_{k\to \infty}\int_{0}^\infty \int (u_k\phi_t - u_k \nab p_k\cdot \nab \phi)\, dx dt + \int u^0_k(x)\phi(x,0)\, dx\\
= \int_{0}^\infty \int (u   \phi_t - u \nab p \cdot \nab \phi)\, dx dt + \int u^0 (x) \phi(x,0)\, dx.\end{multline*}
as desired.

\noindent (iii) Finally,  by passing to the limit $s_k\to 1$ in \eqref{2.9} and using the fact that $C(n,s)$ does not blow up in this limit, we obtain the universal estimate \eqref{linfbound}. \end{proof}

\begin{remark}
The handling of the initial data in the proof above could be replaced by showing that the initial data is taken via a result of continuity of the solution in Wasserstein metric, see Lemma \ref{lemcont} below.
\end{remark}

 When $s_k$ is taken fixed this is the proof of existence of solutions with measure data that improves Theorem 10.1 of \cite{CSV}.

\begin{theorem}\label{exis.fpme}  For every $\mu\in {\cal M}^+(\ren)$, there exists a continuous and nonnegative weak solution of FPME, $0<s<1$, with data $\mu$
 in the  sense of Definition \ref{defsol}. It is a bounded energy solution for all $t\ge \tau >0$.
\end{theorem}


\begin{remark} We will provide in Proposition \ref{pross}  below another approach to the existence of solutions for $s=1$ via the vanishing  viscosity method. \end{remark}


\subsection{Definition of the energy in two space dimensions}\label{n2}

In dimension 2, even though $p=(-\Delta)^{-1} u$ is not in general
of finite energy $\int |\nab p|^2<\infty$, we may still define an
energy by a ``renormalization" procedure. Here we reproduce the
definition of \cite{ams}, which itself follows a similar
construction for Ginzburg-Landau in \cite{os} and \cite{bs}.

The first idea is to define the energy  in a ``renormalized way",
over functions $u$ such that $\int u= k $ with $k$ a fixed constant  (we may do  this without
loss of generality   since the integral is preserved by the
evolution) by
\begin{equation}\label{energy1}
\mathcal{E}(u) = \lim_{R\to \infty} \left( \int_{B(0,R)} |\nab p|^2 - \frac{k^2}{2\pi}\log R\right).
\end{equation}
Contrarily to $\int |\nab p|^2$, it is possible to make this energy
finite even if $k \neq 0$.

Another definition is the following: consider any smooth positive
compactly supported $u_0$ with $\int_{\RR^2} u_0=k$, and let $p_0=
(-\Delta)^{-1} u_0$. For any $u$ with $\int u=k$  such that $w:=
p-p_0$ is in $H^1(\RR^2)$, we define  the energy by
\begin{equation}\label{energy2}
E(u) = \int_{\RR^2} |\nab w|^2 + 2 \int_{\RR^2} u_0 w
\end{equation}
In fact, expanding $w$ as $p-p_0$  and noting that $\int_{B(0,R)}
|\nab p_0|^2 \sim \frac{k^2}{2\pi}  \log R + cst $ as $R \to
\infty$, and using an integration by parts we see that if $w\in
H^1(\RR^2)$, the two definitions differ only by the constant
(depending only on the choice of $u_0$). So, for energy we may
take $E$.

Formally, along the flow, since $w_t= p_t $, we have
\begin{eqnarray*}
\frac{d}{dt} E(u(t)) & = & 2\int_{\RR^2} \nab w \cdot \nab w_t + 2\int_{\RR^2} u_0 w_t\\
& = & - 2\int_{\RR^2} w \Delta p_t + 2  \int_{\RR^2} u_0 p_t\\
& = & 2 \int_{\RR^2}  (w +p_0)   u_t  = 2 \int_{\RR^2} p u_t\\
 & = & - \int_{\RR^2} |\nab p|^2 u. \end{eqnarray*}

An alternate approach would have been to view the $(-\Delta)^{-1}$ operator as a limit of $
(-\Delta + hI)^{-1}$ as $ h \to 0$.


\section{Uniqueness of bounded  solutions for $s=1$}\label{secuniq}
In this section we focus on obtaining a uniqueness result for $s=1$ for bounded solutions. Uniqueness is not expected in all situations, see \cite{ags}.  We also recall that there is no maximum principle so uniqueness cannot follow from comparison arguments.
The proof works only for $s=1$; no  uniqueness proof is known for the FPME in dimensions $n\ge 2$.

  A natural space in which to work for such continuity equations is the space of probability measures endowed with the Kantorovich-Rubinstein-Wasserstein distance (commonly called Wasserstein), indeed the PDE is a gradient flow for the $2$-Wasserstein distance (point of view which was exploited in \cite{AmSr,ams}.
For completeness, we recall (see e.g. \cite{ags}) that the $p$-Wasserstein distance between two probability measures $\mu $ and $\nu$ on $\mr^n$ (and by extension between two positive measures with same mass) is  defined by
\begin{equation}\label{wp}
W_p(\mu, \nu)^p= \inf_{\gamma \in \Gamma(\mu, \nu)} \iint_{\mr^n \times\mr^n} |x-y|^p \, d\gamma(x, y)\end{equation}
where $\Gamma(\mu, \nu)$ is the set of probability measures on $\mr^n \times \mr^n$ whose first marginal is $\mu$ and whose second marginal is $\nu$.
This  inf is achieved by an ``optimal transport plan", and when $1 \le p <\infty$, $W_p$ defines a distance.

\begin{theorem}\label{th31}
There exists at most a unique solution of the initial-value problem for Equation \eqref{eqs1}  in $L^\infty((0, T): L^\infty(\RR^n))$.
\end{theorem}

 When $n=2$, this theorem  recovers with a different method the result on existence and uniqueness of  bounded solutions  of \cite{lz} but without requiring initial compact support.
\begin{proof}The proof is inspired by the method of Loeper \cite{loeper} (see also \cite{AmSr}).
We will need the following result, adapted from \cite[Lemma 8.1]{berma}. The result there is only stated in dimension $n=2$, we include here a proof in any dimension $n$.
\begin{lemma}\label{lem} Let $p= (-\Delta)^{-1} u$. There exists  a constant $C>0$  depending only on the dimension $n$ such that
$$  \forall x, y \in \mr^n,  \quad |\nab p (x)- \nab p(y)|\le C(\|u_0\|_{L^1}+ \|u_0\|_{\infty}) |x-y| \(|\log |x-y|| \mathbf{1}_{|x-y|\le 1/e} +1 \)$$
\end{lemma}
\begin{proof}
Using the integral expression for $\Delta^{-1}$, we have, for any $n\ge 1$,
\begin{equation}
\label{Kp}
\nab p(x)= C_n \int_{\mr^n} K(x-z) u(z)\, dz\qquad K(x)= \frac{x}{|x|^n}.\end{equation}
Let us denote $d=|x-y|.$ First, we treat the case $d\le 1/e$. We may write
$\nab p(x)-\nab p(y)=C_n( I_1 + I_2 +I_3)$ with
\begin{multline*}I_1:= \int_{\mr^n \backslash B(x, 4)} (K(x-z)-K(y-z) ) u(z)\,dz, \quad I_2:= \int_{B(x, 4)\backslash B(x,2d) }( K(x-z)-K(y-z)) u(z)\, dz\\
 I_3:=\int_{B(x, 2d) }( K(x-z)-K(y-z)) u(z)\, dz.\end{multline*}
For  $z\in \mr^n \backslash B(x, 4) $, and since $d=|x-y|\le 1/e$  we have
$|K(x-z)- K(y-z)|\le C|x-y|$ for some constant $C$. Hence we may write
$|I_1|\le C |x-y|\|u\|_{L^1}.$

For $z\in B(x,4) \backslash B(x,2d)$, by a mean value argument, and since $d=|x-y|$, we have
$|K(x-z)-K(y-z)|\le C \frac{|x-y|}{|x-z|^n}. $ We may thus write $$|I_2|\le C|x-y|\|u\|_{L^\infty} \int_{B(x, 4)\backslash B(x,2d)} \frac{dz}{|z-x|^n} = C |x-y|\|u\|_{L^\infty} |\log d|= C\|u\|_{L^\infty} |x-y||\log |x-y||.$$

For $z\in B(x,2d)$     we have $|K(x-z)-K(y-z)|\le |K(x-z)|+|K(y-z)|= \frac{1}{|x-z|^{n-1}}+ \frac{1}{|y-z|^{n-1}} $ and we conclude
$$|I_3|\le \|u\|_{L^\infty} \int_{B(x, 2d)} \frac{dz}{|x-z|^{n-1}} + \frac{dz}{|y-z|^{n-1}} \le C \|u\|_{L^\infty}  d= C |x-y|\|u\|_{L^\infty}.$$
Combining these estimates on $I_1, I_2 $ and $I_3$ we obtain the result in this case.

If $d\ge 1/e$, we split the integral into two parts:
$$I_1:=\int_{\mr^n \backslash B(x, 2d)} (K(x-z)-K(y-z)) u(z)\, dz, \qquad I_2:= \int_{B(x, 2d)} (K(x-z)-K(y-z))u(z)\, dz.
$$
For the first part, we use again the fact that for $z\notin B(x, 2d)$ we have $|K(x-z)-K(y-z)|\le C \frac{|x-y|}{|x-z|^n} \le  C |x-y|$ since $d\ge 1/e$.
We thus find $|I_1|\le C |x-y|\|u\|_{L^1}.$
For $z \in B(x, 2d) $ we argue just as above for $I_3$ and find $|I_2|\le C |x-y|\|u\|_{L^\infty}$.
This concludes the proof.
\end{proof}

We now give the proof of the theorem.
Let us consider  two bounded solutions $u_1, u_2$  on $[0,T]$ with initial data $u^0$, and let $p_1, p_2$ be the corresponding pressures.

We consider the characteristics equation (for $i=1,2$)
\begin{equation}\label{carac}
\frac{dX_i}{dt}(t,x)= - \nab p_i(t,X_i(t,x))\qquad X_i(0, x)=x.\end{equation}
According to \cite{ags} (see also Corollary 3.3 in \cite{loeper}),  if  $u_i$ solve \eqref{eqs1}, then we have
\begin{equation}\label{pf} u_i = {(X_i)}_\# u^0.\end{equation}

We then define
$$Q(t)= \int u^0(x) |X_1(t,x)-X_2(t,x) |^2\, dx.$$

We have that $Q(0)=  0 $, and from \eqref{pf},  $W_2^2(u_1(t), u_2(t)) \le  Q(t)$.

Let us study the evolution of $Q$. Differentiating $Q $ in time yields
$$Q'(t)= - 2 \int u^0(x) (X_1(t,x)-X_2(t,x) ) \cdot (\nab p_1(t, X_1(t,x))- \nab p_2(t, X_2(t,x) )) \, dx $$ hence using the Cauchy-Schwarz inequality
\begin{multline}\label{bdq}
\hal |Q'(t)|\le \( \int u^0 |X_1 - X_2 |^2\, dx \)^{1/2}  \( \int u^0 |\nab p_1(t, X_1)- \nab p_2(t, X_2)|^2\, dx \)^{1/2} \\= Q^{1/2}     \( \int u^0 |\nab p_1(t, X_1)- \nab p_2(t, X_2)|^2\, dx \)^{1/2}.
\end{multline}
On the other hand,  $\( \int u^0 |\nab p_1(t, X_1)- \nab p_2(t, X_2)|^2\, dx \)^{1/2}\le T_1^{1/2} +T_2^{1/2}$ where
$$T_1=  \int u^0 |\nab p_1(t, X_1)- \nab p_1(t, X_2)|^2\, dx $$ and $$T_2= \int u^0 |\nab p_1(t, X_2)- \nab p_2(t, X_2)|^2\, dx.$$

Note first that  from \eqref{pf}
\begin{equation}\label{t2}
T_2= \int u_2(x,t)   |\nab p_1(t,x) - \nab p_2(t,x)|^2 \, dx.\end{equation}
But Theorem 4.4. in \cite{loeper} asserts that, if   $\omega_1$ and $\omega_2$ are two positive measures of same finite total mass and with $L^\infty$ densities, then $\Psi_i=- (\Delta)^{-1} \omega_i$ satisfy
\begin{equation}\label{loep44}
\|\nab \Psi_1-\nab \Psi_2\|_{L^2(\mr^n)} \le \( \max (\|\omega_1\|_{L^\infty}, \|\omega_2\|_{L^\infty} )\)^\hal  W_2(\omega_1, \omega_2)\end{equation}
Applying this to $u_1(t) $ and $u_2(t)$,  using the uniform $L^\infty$ bound on these solutions, and inserting into \eqref{t2}, we find
$$|T_2|\le \max (\|u_1(t)\|_\infty^2, \|u_2(t)\|_{\infty}^2)  W_2^2(u_1(t), u_2(t)) \le C W_2^2(u_1(t), u_2(t)).$$

There remains to bound $T_1$.
Using Lemma \ref{lem},  and denoting $M= \|u^0\|_{L^1}$, we may write \begin{multline*}
|T_1|\le  C (M^2 +  \|u_1(t)\|_{\infty}^2 ) \Big(  \int u^0(x) |X_1-X_2|^2 \log^2 |X_1-X_2|\mathbf{1}_{|X_1-X_2 |\le 1/e}   \, dx  \\+  \int u^0(x) |X_1-X_2|^2 \, dx\Big) .\end{multline*}
Let us define $f(x)= x\log^2 (x)$ for $0\le x\le 1/e$ and $f(x) = 1/e$ for $ x \ge 1/e$. It is easy to check that $f$ is concave, $f(x ) \le x \log^2 x$,   and moreover from the above (and the boundedness of $u_1$)
$$|T_1|\le C \Big(  \int u^0(x) f(|X_1-X_2|^2)  \, dx  +  \int u^0(x) |X_1-X_2|^2 \, dx\Big)$$
Using the concavity of $f$, by Jensen's inequality, we have
 $$ \int u^0(x) f(|X_1-X_2|^2)  \, dx  \le  M f \( \int \frac{u^0(x)}{M} |X_1-X_2|^2 \, dx\)= M
 f(\frac{Q(t)}{M} ) \le C_M Q(t) \log^2 Q(t).$$

 We conclude that  $|T_1|\le CQ(t) (\log^2 Q(t)+1).$
 Combining with the bound for $T_2$, recalling
  that $W_2^2 (u_1, u_2) \le Q$, and
 inserting all these bounds into \eqref{bdq}, we find
$$ Q'(t) \le |Q'(t)|\le C Q(t) \sqrt{1+\log^2 \frac{1}{Q(t)}},$$ where $C$ depends on $M$ and the uniform bounds on $u_1$ and $u_2$.

We may now conclude by a standard integration argument, and since $\int \frac{dv}{v\sqrt{1+\log^2 v}}=  \log (\log v + \sqrt{1+\log^2 v})$ (using a change of variables $u= \log v$), that
 $Q(0)=0$ implies $Q(t)=0$ for $ t\in [0,T]$,  and  since $W_2^2(u_1(t), u_2(t)) \le Q(t) $ and $W_2$ is a distance,   the conclusion follows.
\end{proof}

\section{The universal  bound for $s=1$}\label{sec.ub}

 We now show that the $L^\infty$ norm of solutions  is decreasing in time. Indeed, the
equation can be written as
$$
u_t=-u^2-\nabla u\cdot\nabla p.
$$
which is the sum of absorption and transport with speed ${\bf
v}=-\nabla p$. This expression suggests a universal decay in time of the form
\begin{equation}
\|u(\cdot,t)\|_\infty\le \frac1{t}
\end{equation}
which we will establish below. This agrees   with the limit $s\to 1$
of the smoothing effect of Proposition \ref{pro2.1} for $s<1$, cf. Theorem \ref{thm.exlimit}. What
is surprising is that the bound is universal in the sense that it does not depend on any norm
of the data.  Note that in dimension 2, this bound appears in \cite{lz}.
We have a slightly more precise result in terms of
the $L^\infty$ norm of the data.  The subject was taken up in the recent paper \cite{BLL}, which is  very close to the ideas we explain next in some detail for the reader's convenience.

\begin{theorem}\label{th51} For every weak solution $u$ with bounded initial data $u_0$ we have the estimate
\begin{equation}
u(x,t)\le \frac1{t+\tau}
\end{equation}
where $\tau=1/ \sup_x u_0(x)$. Therefore these solutions coincide (via the uniqueness result for bounded solutions, Theorem {\rm \ref{th31})} with those of Theorem {\rm \ref{thm.exlimit}.}
\end{theorem}
\begin{remark}
 A similar result holds for the lower bound of the negative values
when dealing with signed solutions. \end{remark}

  We start by studying a viscous approximation of the equation (which, by the way, is like the repulsive version of the Keller-Segel equation for chemotaxis \cite{KS71, JL92}).
  \begin{proposition}\label{pross}
Let $u^0\in L^\infty(\mr^n)\cap L^1(\mr^n)$ with $u^0\ge 0$. There exists a unique solution to
  \begin{equation}\label{viscousapp}
 \left\{\begin{array}{ll}
  \p_t u_\ep = \ep \Delta u_\ep+ \nab \cdot (u_\ep \nab p_\ep) & \qquad p_\ep= (- \Delta)^{-1} u_\ep\\
  u_\ep(\cdot, 0)=u^0 &\end{array}\right.\end{equation}
  Moreover $u_\ep$ is smooth, satisfies
  \begin{equation}\label{boundvisc}
  0 \le u_\ep(x,t) \le \|u^0\|_{L^\infty}\end{equation}
  and converges locally in $L^q(\mr^n \times [0, +\infty)$ for any $1\le q<\infty$, to $u$, the unique solution to \eqref{eqs1}--\eqref{init}.
  \end{proposition}
  We note that this ``vanishing viscosity approach" is an alternative way of obtaining existence of solutions to \eqref{eqs1}--\eqref{init}.

  \begin{proof}
  This natural ``viscosity approximation''  was previously considered in Sch\"atzle-Styles \cite{sch-styles}, for a very similar equation, but in a bounded domain of $\mr^2$. Because the proof is tailored for a bounded domain,  it is not immediate to adapt it to our case.   In order to prove the  existence of solutions to \eqref{viscousapp}, we set up a fixed point argument.

Take  $T>0$ and let
 $$K=\left\{ v(x,t) \in L^\infty (\mr^n\times(0,T)), \ v \ge 0\,,  \text{and} \  \forall t \in [0,T], \int_{\mr^n} v(x,t) \,dx \le \int_{\mr^n} u^0(x)\, dx\right\}.$$
  $K$ is clearly a closed convex subset of the Banach space $L^\infty(\mr^n \times [0,T])$.

  \begin{lemma} Let $u^0 \in L^\infty(\mr^n)\cap L^1(\mr^n)$. For any $v \in K$, letting $p_v = (-\Delta)^{-1} v$. Then, we have
  $\nab p_v \in L^\infty(\mr^n \times[0,T])$
  and there exists a unique solution $u:={\cal T}(v)$  over $[0,T]$  of the problem
  \begin{equation}
  \label{tvu}
\left\{  \begin{array}{ll}
  \p_t u = \ep \Delta u + \nab \cdot (u \nab p_v)\,,\\
  u(\cdot, 0)=u^0\,. \end{array}\right.
  \end{equation}
  That solution satisfies ${\cal T}(v) \in K$ and $\|{\cal T}(v)\|_{C^{0, \alpha}} \le C.$ ${\cal T}$ is continuous and compact from $K$ to itself.
  \end{lemma}
  \begin{proof}
First we note that if $v \in K$, then for all $t \in [0,T]$, $v(\cdot,t) \in L^1\cap L^\infty$, which implies that $\nab p_v \in L^\infty(\mr^n \times [0,T])$ by Nash-Sobolev embedding, or more precisely, by observing
\begin{multline*}|\nab p_v|(x,t) \le \int_{\mr^n} \frac{v(y,t)}{|x-y|^{n-1}}\, dy \le \|v\|_{L^\infty(\mr^n \times [0,T])}
\int_{B(x,1)} \frac{dy}{|x-y|^{n-1}} + \int_{\mr^n} v(y) \, dy\\
 \le C\(\|v\|_{L^\infty(\mr^n\times [0,T])} + \|v(\cdot, t)\|_{L^1(\mr^n)}\),\end{multline*}
where $C$ depends only on the dimension. In fact, direct inspection allows to prove in addition that $\nab p_v(\cdot, t) \in C^{0,\a}(\mr^n)$ for some $\a>0$.

We are thus led to solving a linear equation of the form
 \begin{equation}
  \label{tlinear}
\left\{  \begin{array}{ll}
  \p_t u = \ep \Delta u + \nab \cdot (u f)\\
  u(\cdot, 0)=u^0 \end{array}\right.\end{equation}
  with $f\in L^\infty(\mr^n \times [0,T], \mr^n)$.
  This can be solved by a standard convolution argument, see for example Lemma 12 in \cite{bdp} which applies exactly (the fact that the dimension is $n$ and not $2$ does not affect at all the argument), and it provides us with a unique solution ${\cal T}(v) \ge 0$ defined for all positive times, hence in $[0,T].$
  The fact that for all $t\in [0,T]$,  $\int_{\mr^n} {\cal T}(v)\, dx  \le \int_{\mr^n} u^0$ follows by integrating in time the equation.
On order to show  that ${\cal T}(v) \in K$, there remains to show that ${\cal T}(v) \in L^\infty(\mr^n\times [0,T]) $.

  Let us consider  $u = {\cal T}(v)$ solving \eqref{tlinear} and  let $\phi: \mr \to \mr$ be a smooth function. Using   \eqref{tlinear}, integration by parts and $\Delta p_v=-v$, we have
  \begin{multline}
  \frac{d}{dt} \int_{\mr^n} \phi(u) = \int \phi'(u) \( \ep \Delta u + \nab \cdot (u \nab p_v)\)\\
 = - \ep \int \phi''(u) |\nab u|^2 - \int  u \phi''(u) \nabla p_v \cdot \nab u \\
 = -\ep \int \phi''(u) |\nab u|^2 - \int \nab p_v \cdot\nab [H(u)]\\
  = -  \ep \int \phi''(u) |\nab u|^2 - \int v  H(u)\,,\\
  \end{multline}
where we define $H(s)= \int_0^s t\phi''(t)\, dt.$
If in addition  $\phi $ is  convex it follows that  $H\ge 0$ and, since $v \ge 0$, we find
$$
\frac{d}{dt} \int_{\mr^n} \phi(u) \le 0.$$
Taking $\phi$ nonnegative convex and  supported over $[\|u^0\|_{L^\infty}, + \infty)$ we conclude that $0 \le u(x,t) \le \|u^0\|_{L^\infty}$ for every $(x, t) \in \mr^n \times[0,T]$. This proves that ${\cal T}(v)\in L^\infty$ and it belongs to $K$. ${\cal T} $ thus maps $K$ into itself. It is also easy to see that it is continuous.

Now, standard regularity theory  (cf. \cite{LSU})  gives us Schauder estimates
$$\|{\cal T}(v)\|_{C^{0, \a}(\mr^n \times [0,T])}\le C \(\|v\|_{L^\infty(\mr^n\times [0,T]) }+ \|v\|_{L^1(\mr^n\times [0,T])}\)\le CT\(\|u^0\|_{L^\infty} +\|u^0\|_{L^1}\) .$$ Since $C^{0,\a}$ embeds compactly into $L^\infty$, ${\cal T} (K)$ is thus precompact.
\end{proof}

We may now apply Schauder's fixed point theorem \cite{lieb}  which gives us the existence of a fixed point for ${\cal T}$ hence a solution to \eqref{viscousapp}, satisfying $u_\ep \ge 0$ and $\|u_\ep\|_{L^\infty(\mr^n \times [0,T])} \le \|u^0\|_{L^\infty}$.

Finally once we have established \eqref{boundvisc}, convergence of $u_\ep $ to a solution of
\eqref{eqs1} can be proven following the same steps as in the proof of Theorem \ref{thm.exlimit}
(except replacing $s_k$ by $1$). The solution is unique by Theorem \ref{th31}, which implies convergence
of the whole sequence $u_\ep$.
\end{proof}

\begin{proof}[Proof of Theorem \ref{th51}]
  We first prove  the desired estimate for $u_\ep$ solution to \eqref{viscousapp}, which is smooth by standard parabolic regularity theory.
  Let us consider $\phi$ smooth such that $\phi(u)=0$ for  $u \le 0$, $\phi''\ge 0 , \phi' \ge 0$,  and let us
   compute $\dfrac{d}{dt} \dint \phi(u_\ep-M(t))dx$, with $M$ some smooth nonnegative function of time. This yields
\begin{equation}\label{ddtp}
\dfrac{d}{dt}\dint \phi(u_\ep-M(t))dx=\dint \phi'(u_\ep-M(t))\p_t u_\ep\,dx -M'(t)\dint \phi'(u_\ep-M(t))\,dx.
\end{equation}
Using  the equation, we find, after several integrations by parts
\begin{multline*}
\dint \phi'(u_\ep-M(t))\p_t u_\ep
 = \ep \int \phi'(u_\ep- M(t)) \Delta u_\ep - \int \phi''(u_\ep- M(t)) \nab u_\ep \cdot u_\ep \nab p_\ep\\
 = -  \ep \int \phi''(u_\ep- M(t)) |\nab u_\ep|^2 -
 \int \phi''(u_\ep - M(t)) (u_\ep - M(t))  \nab u_\ep \cdot \nab p_\ep
 \\-M(t) \int \phi''(u_\ep -M(t)) \nab u_\ep \cdot \nab p_\ep\\
 \le
- \int \nab [ H(u_\ep -M(t))] \cdot \nab p_\ep   - M(t) \int \phi''(u_\ep -M(t)) \nab u_\ep \cdot \nab p_\ep.   \end{multline*}
 Here, we have defined again
 $H(s)= \int_0^s t\phi''(t)\, dt.$  Using  $-\Delta p_\ep= u_\ep$ and another integration by parts, we find that the first term on the right-hand side is $ - \int H(u_\ep -M(t)) u_\ep \le 0$. There remains to estimate  in a similar way
 \begin{multline*}-
 M(t) \int \phi''(u_\ep -M(t)) \nab u_\ep \cdot \nab p_\ep
= - M(t) \int \nab[\phi'(u_\ep - M(t))]  \cdot \nab p_\ep \\
=- M(t) \int \phi'(u_\ep -M(t)) u_\ep \le -  M(t)^2 \int \phi'(u_\ep - M(t)) \,,
\end{multline*}
because $\phi'$ is supported on the positive part of the real line and $M(t) \ge 0$.
We conclude that $$\dint \phi'(u_\ep-M(t))\p_t u_\ep\,dx  \le  -  M(t)^2 \int \phi'(u_\ep - M(t))\, dx$$ and combining with  \eqref{ddtp}
\begin{equation}
\dfrac{d}{dt}\int \phi(u_\ep-M(t))dx \le - (M^2(t)+M'(t)) \int \phi'(u_\ep - M(t))\,dx .
\end{equation}
Choosing $M(t)= \frac{1}{t+\tau}$ with $\tau = (\sup u_0(x))^{-1}$, we have $M'(t)= - M^2(t)$, and it follows that   \
$\dfrac{d}{dt}\dint \phi(u_\ep-M(t))\, dx \le 0$. On the other hand, since $u^0\le M(0)$ we have $\int \phi(u^0 - M(0)) \, dx\le 0$, and thus we have obtained that $\int \phi(u_\ep - M(t))\, dx \le 0$ for all $ t\ge 0$.

Taking now $\phi=\phi_n$ to be a sequence that approximates the function $\phi_0(x)= sgn_+(x)$.  we obtain $u_\ep(t) \le M(t)$ a.e. for $t\ge 0$.
  Passing to the limit $\ep \to 0$, and using Proposition \ref{pross},  we arrive at the result.
  \end{proof}

\begin{remark} The estimate $|u(x,t)|\le \frac{1}{t+\|u_0\|_{\infty}^{-1}}$ also applies  to solutions of
every sign which generalizes the case of ``dissipative weak solutions" studied in \cite{mz}.\end{remark}

We conclude this section with a result of continuity of the solutions in Wasserstein distance, which we do not use however.

\begin{lemma}\label{lemcont}
The solutions to \eqref{eqs1} with $u^0 \in \mathcal{M}^+(\mr^n)$  satisfying \eqref{linfbound} are ``absolutely continuous" (hence continuous) for the $W_q$-Wasserstein distance for any $q\in [1, \infty)$: more precisely $W_q(u(t_2), u(t_1)) \le \int_{t_1}^{t_2} m(s)\, ds$ with $m\in L^1([0,T])$ for any $T>0$.
\end{lemma}
\begin{proof}By conservation of mass,  for every $t>0$ we have the two a priori estimates $\|u(t)\|_1\le M$, $\|u(t)\|_\infty \le 1/t$. By the Nash-Sobolev embeddings we then get $\nabla p(\cdot,t)\in L^\infty(\ren)$. A simple scaling argument  then shows that
\begin{equation}\label{bornev}
\|\nabla p(\cdot,t)\|_\infty \le C \|\Delta p(\cdot,t)\|_1^{1/n}\|\Delta p(\cdot,t)\|_\infty^{(n-1)/n}\le C_1 M^{1/n}\,t^{-(n-1)/n}
\end{equation}
The claimed continuity is a direct consequence of the nature of the equation \eqref{eqs1} as a continuity equation, i.e. of the form $ \p_t  u + \nab \cdot (u \mathbf{v})=0$. Theorem 8.3.1 in \cite{ags} asserts that for such an equation, if $\int_0^T \|\mathbf{v}(t)\|_{L^q(u(t)) }\, dt <+\infty$ then the solution $u(t)$ is absolutely continuous in $[0,T]$ with values in the space of measures equipped with the distance $W_q$. The precise definition of this (cf. \cite{ags}) is that $W_q(u(t_1), u(t_2))\le \int_{t_1}^{t_2} m(s)\, ds $ with $m \in L^1([0,T])$, and it implies continuity in the usual sense.
In the present case, $\mathbf{v}= \nabla p$ and, from \eqref{bornev} we have \begin{multline*}
\int_0^T \|\mathbf{v}(t)\|_{L^q(u(t))}\, dt
= \int_0^T\( \int_{\mr^n} |\nab p(t)|^q u(t)\)^{1/q}\, dx\, dt\\
\le \int_0^T C t^{1/n-1} \(\int_{\mr^n} u(t) \, dx\)^{1/q} \, dt\le C' M^{1/q}T^{1/n}.\end{multline*}
It follows that the result of \cite{ags} applies and $u(t)$ is  continuous for the Wasserstein distance $W_q$.
\end{proof}


\section{Radial solutions. Burgers' equation}
\label{sec.radial}
In this section we analyze radially symmetric solutions and exhibit the special class of vortex-patch solutions.  We use the Burgers' equation approach via the mass transform, which already appeared in \cite{bgl} and has precedents in chemotaxis problems, \cite{JL92}. The main purpose of most of this section is to fix the ideas and notation that we will use in the last two subsections, and in the asymptotic analysis of Section \ref{ss.abrad}.

\begin{lemma}
\label{burg}
Let $u$ be a solution to \eqref{eqs1} with  radially symmetric bounded initial data $u_0(r)$, $r=|x|$.
Then $u$ is radial, and the ``mass function" defined by
\begin{equation}
\label{M}M(r,t)= \int_0^r s^{n-1} u(s,t)\,ds
\end{equation}
solves the Burgers' equation
\begin{equation}\label{burger}
M_t+  \frac1{r^{n-1}} M M_r =0\,,
\end{equation}
 which after the change of variables $s=r^n/n$ can be written in the standard form
$M_t+M \,M_s=0$.
\end{lemma}

\begin{proof}This follows from elementary computations, or can be found in \cite{bgl}.
   \end{proof}
   Note that \eqref{burger} is now a local equation, contrary to \eqref{eqs1} and  that it   looks very much like the equation for  chemotaxis, but with the opposite sign.   The speed of movement is $M/r^{n-1}=v\ge 0$. Notice that
\begin{equation}\label{6.2b}
M_t\le 0.\end{equation}

A common problem with Burgers' equations and other first order conservation laws is that they may develop shocks in finite time and the uniqueness of solutions needs the concept of entropy solutions, cf. \cite{Lax57, Kr70}. This is not the case here, due to  the fact that $M(r,t)$ is a monotone nondecreasing function of $r$ and that the velocity $v$ is positive, so that the characteristic lines spread. Conversely, it is easy to see that weak solutions $M(s,t)$ of the Burgers equation with $M_s\ge 0$ bounded, give rise to nonnegative, bounded, radial, weak solutions of the hydrodynamic equation $u_t=  \nabla (u\nabla (-\Delta)^{-1}u)$.

Burgers' Equation \eqref{burger} has the property of comparison (maximum principle) in the class of nonnegative and monotone solutions $M$.  Since the equation is also 2-homogeneous in $M$, B\'enilan-Crandall's result \cite{BenCr81} applies to give the pointwise inequality
\begin{equation}\label{burg.hom.inq}
M_t\ge-\frac{M}{t}.
\end{equation}
Using the equation this implies the further inequality (that we already know)
\begin{equation}\label{burg.hom.inq2}
 \ M_r\le \frac{r^{n-1}}{t}\,, \qquad \text{i.\,e.} \quad  u(x,t)\le \frac{1}{t}.
\end{equation}
These inequalities are optimal as we will show by means of the examples of the next section.

The approach via the mass function $M(r,t) $ will be used below to construct interesting classes of explicit solutions, that serve as prototypes or as counter-examples of the theory.

\subsection{Solutions in dimension $n=1$}

The Burgers' transformation, i.\,e., the formulation in terms of the mass function,  applies easily when we work in dimension $n=1$ and we do not need the condition of radial symmetry to get in a direct way the standard Burgers' equation
\begin{equation}\label{burger.d1}
M_t+   M M_r =0,
\end{equation}
where we are still assuming that we deal with nonnegative solutions $u\ge 0$.

\subsection{The elementary vortex patch}\label{sec.evp}
We next recover the existence of the vortex patch self-similar solution.
\begin{proposition}
\label{vortexpatch} The equation  \eqref{eqs1}
admits the family of   radial ``vortex patch solutions" of the form
\begin{equation}\label{vpatch1}
u(x,t) = \frac{1}{t+\tau}\chi_{ |x|\le R(t+ \tau)^{1/n}} \end{equation}
(as well as all their translates) where $R>0 , \tau\ge 0$ are free parameters.
\end{proposition}

We skip the proof of this known result,  see for instance \cite{BLL}. This configuration is called in the terminology of transport equations a {\sl rarefaction wave.}  The simplest form of the solution is
\begin{equation}\label{vps}
u(r,t)=\frac1{t} \chi_{r \le Rt^{1/n}  } \end{equation} which formally corresponds to $\tau =0$.
We call this solution the {\sl elementary vortex patch.}  Notice that it is a solution in the sense of Definition \ref{defsol}.  It has a Dirac mass as initial data for the density, hence it qualifies as the {\sl fundamental solution} of the problem.


All  of the above  solutions can also be  described as {\sl expanding mesas.} The use of the term ``mesa'' appears  in some limit cases to nonlinear diffusions, see e.g.  \cite{EHKO} and \cite{FdmHo} or \cite{Vapme}.

\medskip

\noindent {\bf Largest solution.} Note that in the limit when the initial mass goes to infinity we obtain the solution
\begin{equation}
u(x,t)=\frac1{t}, \qquad M(r,t)={r^{n}}{nt}
\end{equation}
that does not have an interface. This  solution is an absolute a priori upper bound for all
solutions of our problem with $s=1$. It satisfies equations  \eqref{burg.hom.inq} and \eqref{burg.hom.inq2} with equality.

\subsection{Self-similar solutions}\label{secselfsim}

The vortex patch solution can be obtained as the limit of interesting self-similar solutions of the FPME with $0<s<1$. Such solutions are constructed in the paper \cite{CV2} as self-similar solutions of the form
\begin{equation}\label{selfsim}
\sl  u(x,t)=t^{-\alpha} U(x/t^{\beta})
\end{equation}
In order to cancel the factors including $t$ explicitly, we  get the condition on the exponents
\begin{equation}
\sl  \alpha + (2-2s)\beta =1\normalcolor
\end{equation}
 If we also impose conservation of (finite) mass, which amounts to the condition
$\alpha= n \beta$, we arrive at  the precise value for the exponents:
$$
\sl  \beta=1/(n+ 2-2s), \quad \alpha=n/(n+ 2-2s).\normalcolor
$$
 We also arrive at the  {\sl nonlinear, nonlocal elliptic equation} for the self-similar profiles
\begin{equation}\label{baren.eq}
\sl  \nabla_y\cdot(U\,\nabla_y (P+ a |y|^2))=0, \quad \mbox{with \
} \ P={\cal K}_s U.\normalcolor
\end{equation}
where $a=\beta/2$, and $\beta$ are defined just above. The solution of equation \eqref{baren.eq} is obtained in \cite{CV2} by reducing it to an obstacle problem for the $s$-Laplacian operator. On the other hand,  Biler, Imbert and Karch \cite{bik}  have obtained the explicit formula for a self-similar solution in the form
\begin{equation}\label{baren.s}
u(x,t;C_1)=t^{-\alpha} (C_1-k_1\,x^2t^{-2\alpha/n})_+^{1-s}
\end{equation}
with $k_1, C_1>0$ and $\alpha=n/(n+2-2s)$ as before. These solutions bear a resemblance  to those  of the standard porous medium equation \cite{Vapme,BKM,JLVAbel}.  This is why they have been termed {\sl Barenblatt solutions for the FPME.}

\medskip

\noindent {\bf Selfsimilar solution for $s=1$:}  This approach  works for $s=1$ and gives  $\alpha=1$, $\beta=1/n$, where we replace ${\cal K}_s$ by ${\cal
K}_1=(-\Delta)^{-1}.$
 We then recover the vortex patch solution \eqref{vps}.   It is also easy to see that the family of self-similar solutions \eqref{baren.s} pass to the limit $s\to 1$ to produce the vortex patch solution \eqref{vps}  (after careful selection of the constants). This is an explicit example of application of Theorem \ref{thm.exlimit}.

\medskip

\noindent {\bf Comment.} For the reader's convenience,  we give more details in the case $s=1$.  The solution of the obstacle problem has to satisfy $U=-\Delta P= 2na= n\beta=1$ in the
coincidence set $|y|\le R$, while $\Delta P=0$ outside (for
$|y|\ge R$). Therefore, in the outer region
$$
P=D\,|y|^{2-n}, \quad D>0.
$$
$C^1$ agreement of the two expressions at $y=r$ gives
$D=R^n/n(n-2)$ and $ R^2=2(n-2)C,$ and we recover the solution \eqref{vps}.


\subsection{Other examples of vortex patches}
\label{sec.radial2}

We may next construct many solutions of the radial problem as rarefaction waves, which are  not very different from the elementary vortex patch. We take an initial configuration formed by a ball patch at the origin plus an annular patch around, i.e.,
$$
u_0(r)=c_1 \quad \text{ for} \ 0\le r \le R_1, \qquad
u_0(r)=c_2 \quad \text{ for} \ R_2\le r \le R_3,
$$
with $c_1,c_2>0$ and $0<R_1<R_2<R_3$, $u_0(r)=0$ otherwise. Then
\begin{align*}
M(r,0) & =  (c_1/n) r^n\quad   &   \text{for } 0\le r\le R_1\\
&=   (c_1/n) R_1^n  &  \text{for} \  R_1\le r\le R_2\\
&=   (c_2/n) r^n + (c_1/n) R_1^n - (c_2/n) R_2^n  &  \text{for } R_2\le r\le R_3 \\
 &=    (c_2/n) (R_3^n - R_2^n)  + (c_1/n)  R_1^n & \text{for} \ r\ge R_3
   \end{align*}
Note that the continuity of $M$ requires  that we choose $c_2R_2^n=c_1R_1^n$.
At later times the solution keeps the same shape, i.e. the support of $u(r,t)$ is still the union of a ball and an annulus (both expanding).
$u(r,t)$ and  $M(r,t)$ are   explicit, in regions that are separated by interfaces: \\(i) The first region is an inner ball where $u$ and $M$ have the form
$$
u_1(r,t)=\frac{c_1}{c_1 t+1}=\frac{1}{t+\tau_1}; \qquad
M_1(r,t)=\frac{c_1 r^n}{n(c_1t+1)}= \frac{r^n}{n(t+\tau_1)}
$$
($\tau_1=1/c_1$). The bounding circle (centered at the origin)  has radius $S_1(t)$ satisfying $S_1'(t)=v(S_1(t),t)=M(S_1(t),t)/S_1(t)^{n-1}$
with $S_1(0)=R_1$, hence
$$
S_1(t)= R_1(c_1t+1)^{1/n}= R_1c_1^{1/n}(t+\tau_1)^{1/n}.
$$
(ii) There is then an annular region with $M(r,t) =c_1R_1^n/n$ and $u(r,t)=0$, bounded by this interface and the circle of radius
$$
S_2(t)=R_1(c_1t+ (R_2/R_1)^n)^{1/n}=R_2(c_2t+ 1)^{1/n}=R_2c_2^{1/n}(t+\tau_2)^{1/n}, \quad \tau_2=(R_2/R_1)^n\tau_1.
$$
(iii) There is then an annular  region of positive density (initially $R_2<r<R_3$) given by
$$
u_2(r,t)=\frac{c_2}{c_2 t+1}=\frac{1}{t+\tau_2}; \qquad
M_2=\frac{c_2 r^n}{n(c_2t+1)}= \frac{r^n}{n(t+\tau_2)}
$$
where $\tau_2=1/c_2>\tau_1$. The outer interface of this annulus
is  the circle of radius $S_3(t)=R_3 (c_2 t+1)^{1/n}=R_3 c_2^{1/n}(t+\tau_2)^{1/n}$. \\
(iv) Finally,  in the exterior region, complement of the ball of radius $S_3(t)$, $M(r,t) $ is  constant $M$ and $u$ is zero.

\medskip

\noindent {\bf Conclusion.} In this example there are two connected components for the density that expand out with time. The  inner patch  is unperturbed by the outer mass, while  the outer annular patch is spreading out  in a self-similar way. The density in both  patches is constant in space  but the density outside depends on the density inside, time,  and the ratio $R_2/R_1$. Both densities approach $1/t$ in first approximation as $t\to\infty$. Moreover, the gap between the two patches is
$$
d(t)= C\,[(t+\tau_2)^{1/n}-(t+\tau_1)^{1/n}]
$$
which relatively goes to zero  like \ $O(t^{-(n-1)/n})$ as \ $t\to \infty$.


\subsection{Lack of comparison principle}

This example allows to construct a counterexample to comparison of densities. Let us take for first  solution $u_1(r,t)$  the above example with parameters $0<R_1<r_2<R_3$.  For second solution $u_2(x,t)$ let us  take the elementary vortex patch with center $x_0$ such that $|x_0|=(R_2+R_3)/2$,  with height equal or less than $c_2$ and initial radius $R_4$ much less than $(R_3-R_2)/2$. In this way we have $u_1(x,0)\ge u_2(x,0)$. However, the support of $u_2$ expands from the initial ball $B_{R_4}(x_0)$ while the support of $u_1$ has a ``hole" that eventually will  ``cover"  that ball at some time. Hence comparison of densities does not persist in time.

Note that the lack of a  comparison of densities principle was shown for the solutions of the FPME with $s>1/2$ in \cite{CV1}. The argument there was less explicit due to the lack of explicit solutions.

Recall that for radial solutions  on the other hand we have comparison of the mass functions $M(r,t)$.  A similar result for FPME is known in dimension $n=1$ where the approach of integrating in space is viable, \cite{BKM}.

\section{Solutions with compact support}\label{sec.compact}

We now  return to non necessarily radial solutions to \eqref{eqs1} but impose the condition of compact support on the data. This short section is essentially contained in \cite{BLL}, where more details can be found. We have kept the proof since some of the arguments we explain here are used in what follows.

It is known \cite{CV1} that the solutions of the FPME for
$0<s<1$ have the property of finite propagation so that initial
data with compact support produces solutions with the same property
for every fixed $t>0$. We can check that the calculation of the
supersolution used to prove the property for $s<1$ can be
extended to $s=1$. However, a comparison argument with explicit
solutions is easier in this case, by using the selfsimilar
solutions  constructed in Section \ref{sec.evp}. It works as follows.

\begin{theorem}\label{thm.7.1} Given  bounded initial data $u_0$ with compact
support in the ball of radius $R_0$, then for every $t>0$ the
solution $u(\cdot,t)$ is supported in the ball of radius
$R_0(1+ \|u_0\|_{\infty} t)^{1/n}$.
\end{theorem}

 \begin{proof} We consider the characteristics that according to \cite{ags} transport the solution. Using the equation we see that the speed of such lines is given by
\begin{equation}
v(x_0,t)=\nabla p(x_0,t)=(\nabla (- \Delta)^{-1}u)(x_0,t)=C_1\int \frac{u(y,t)(x_0-y)}{|x_0-y|^{n}}\,dy
\end{equation}
where $C_1=1/((n-2)\omega_n)$, with $\omega_n$ the volume of the unit ball in $\ren$.
 So it is bounded at all $x$ and $t$ and we conclude that the initial support propagates only to a finite distance in any finite time interval.

Take now $t>0$ fixed and assume that we are at $x_0$, a point on the free boundary or boundary of the support $\Omega(t)=\overline{ \{x: u(x,t)>0\}}$.
We consider an extreme point, i.\,e., such that  $|x_0(t)|=\sup\{|x|: x\in \partial\Omega(t)\}$. We may assume after rotation that $x_0=(R(t),0,\cdots,0)$. We calculate the velocity along the radius direction (which is the velocity with which the boundary separates from the origin). Then,
\begin{equation}
v_1(x_0,t)=C_1\int \frac{u(y,t)(R(t)-y_1)}{|x_0-y|^{n}}\,dy,
\end{equation}
The integrand has a positive sign for every $y\in \Omega(t)$. Therefore, the worst case is when we
replace $u$ by its maximum estimate given by Theorem \ref{th51}, to get
$$
|v_1(x_0,t)| \le C_1\frac1{t+\tau}\int_{B_R(0)} \frac{(R-y_1)}{|x_0-y|^{n}}\,dy,
$$ where $\tau = 1/\|u_0\|_{\infty}$.
Using scaling we get
$$
|v_1(R(x_0,t)| \le C_1\frac{R(t)}{t+\tau}\int_{B_1(0)} \frac{(1-y_1)}{|{\bf e_1}-y|^{n}}\,dy,
$$
with ${\bf e_1}=(1,0,\cdots,0)$.
Next, we notice that this is the integral that would be obtained for the    self-similar solution of Proposition \ref{vortexpatch}.  After a calculation in radial coordinates we have
$$
R'(t)= |v_1(x_0,t)|\le \frac{R(t)}{n(t+\tau)}
$$
which, integrated in time, yields the estimate $R(t)\le R_0  \(\frac{t+\tau}{\tau}\)^{1/n}$, i.e. the result.
 \end{proof}

\begin{remark} We have really proved the following:
Let $u$ be a bounded  nonnegative weak solution.
Assume the support of the  initial data is contained in a ball of radius $R_0$. Let $U$ be the vortex-patch
solution with same initial $L^\infty$ norm and initial support $B_{R_0}(0)$. Then the support of
$u$ at time $t$ is contained in the support of $U$ at time $t$.
\end{remark}


 We conclude the section with a similar statement for solutions which are not necessarily initially bounded.
\begin{proposition} For every initial measure $\mu\in {\cal M}^+(\mr^n)$ with compact support, there exists a  weak energy solution with compact support at all times and satisfying  a bound of the type \eqref{linfbound}.\end{proposition}

\begin{proof}  The existence of a solution satisfying \eqref{linfbound}  is proven in Theorem \ref{thm.exlimit}.     To check it has compact support, let us   approximate first the initial data by smooth initial data with same compact support and same bound $M$ for their total mass $\int u_0\,dx$.  Then, by conservation of mass,  for every $t>0$ we have the two a priori estimates $\|u(t)\|_{L^1}\le M$, $\|u(t)\|_\infty \le 1/t$. By the Nash-Sobolev embeddings we then get $\nabla p(\cdot,t)\in L^\infty(\ren)$. A simple scaling argument  then shows that
\begin{equation}\label{bornevv}
\|\nabla p(\cdot,t)\|_\infty \le C \|\Delta p(\cdot,t)\|_1^{1/n}\|\Delta p(\cdot,t)\|_\infty^{(n-1)/n}\le C_1 M^{1/n}\,t^{-(n-1)/n}
\end{equation}
Using the  description by characteristics (valid for the approximating solutions, as we saw in the previous theorem)
and integrating, it follows that  the supports of the approximating solutions (compact by the previous theorem) are included in a ball of radius
\begin{equation}\label{form.fb}
R(t)\le R_0+ C_2M^{1/n}t^{1/n}\,,
\end{equation}
where $C_2$ is independent of the approximation (and neither does $R_0$).  The conclusion follows by passing to the limit.  Note that we estimate the growth of the support like $O(t^{1/n})$ as expected in the best case.
\end{proof}


\section{Asymptotic behaviour}\label{sec.asbeh}

The class of nonnegative solutions with finite mass has a property of simple asymptotic behaviour of self-similar form that is typical of diffusion processes, and also of some conservation laws. We want to prove that the self-similar vortex patch solutions of Section \ref{sec.evp} are attractors of the finite-mass solutions. This has been proven in \cite{BLL} in the case of compactly supported initial data. We extend it here to solutions with finite second moment and to all radial solutions.

\subsection{Asymptotics  by the entropy method}\label{sec.as.gen}

By using the entropy method we can prove asymptotic behaviour for solutions with general data, with the only extra requirement  that the second moment of the initial mass distribution has to be finite. The proof follows the outline of the one done for the diffusive process with $s<1$ in \cite{CV2}.

In order to consider the question of large-time behaviour it is better to switch to the renormalized flow.  More specifically, we introduce a change of the space, time and density variables corresponding to the expected large time behaviour, as follows:
\begin{equation}\label{RF}
y=x/(t+1)^{1/n}, \quad \tau=\log (1+t), \quad  U(y,\tau)=(t+1)\,u(x,t)\,,
\end{equation}
as well as the corresponding change for pressure and velocity:
$$
 P(y,\tau)=(t+1)^{(n-2)/n}p(x,t), \quad V(y,\tau)=(t+1)^{(n-1)/n}v(x,t)=-\nab P(y, \tau).
$$
As a model for the asymptotic behaviour we consider the elementary vortex of Subsection \ref{sec.evp} with initial height 1, i.\,e.\,,
$$
 u_*(x,t)=\frac{1}{t+1} \qquad \mbox{ for } \quad x\in B_R(t),\quad R(r)=R_0\,(t+1)^{1/n}.
$$
and $u_*(x,t)$ zero otherwise. We further choose $R_0$ so that the total mass $M$ is the same:
$$
\int u_*(x,t)\,dx=\int u_0(x,t)\,dx,
$$
(i.\,e., $\omega_n R_0^n=\|u_0\|_1=M$). The corresponding  renormalized density is
\begin{equation}\label{ustar}
U_*(y,\tau)= \chi_{|y|\le R_0},\end{equation}
and the corresponding  pressure, defined by $ P_*(\tau, y)=(-\Delta)^{-1}U_*(\tau, y)$ satisfies  $$P_*(\tau, y)=\(C-(y^2/2n)\) \ \text{and}  \ - \nabla P_*(\tau, y)= \frac{y}{n} \ \text{ in} \  B_{R_0}(0).$$\normalcolor
In the proofs below we will consider the solution $U(y,\tau)$ of the renormalized flow   associated to a solution  $u(x,t)$ via \eqref{RF}. Some direct  calculations show that  $U$ solves
\begin{equation}\label{ren.eq}
U_\tau=\nabla \cdot [U ( \nabla P+y /n )], \qquad \Delta P+ U=0,
\end{equation}
where $\nabla$ and $\Delta$ indicate now differentiation w.r.t. $y$. Note that this is a mass conservation law
$$
U_\tau+\nabla \cdot (U \vec{v})=0, \qquad \vec{v}=-\nabla P - y/n,
$$
where the velocity  is the sum of the renormalized particle velocity and the confining velocity $-y/n$ produced by the change of  variables.

We next state the asymptotic behaviour result.

\begin{theorem} Let $n\ge 2$  and assume that $u_0\in L^1(\mr^n)$, $u_0\ge 0$, and  $u_0$ has a finite second moment, $\int_{\mr^n} y^2\,u_0(y)   \,dy< \infty$. Then,  the speed and density converge to equilibrium after normalization, as follows:\\
(i) $U$ stabilizes  towards $U_*$, the elementary patch with the same mass,  in the sense that
\begin{equation}\label{as.from.1}
\lim_{\tau\to\infty} \|U(y,\tau)-U_*(y)\|_1=\lim_{t\to\infty} \|u(\cdot,t)-u_*(\cdot,t)\|_1=0\,.
\end{equation}
We also have  strong  convergence $U(y,\tau)\to U_*(y)$ in all the $L^p(\ren)$ spaces, $1\le p<\infty$, and weak-* convergence in $L^\infty(\ren)$.

\noindent(ii) The corresponding velocities  converge
\begin{equation}\label{asymp.energy}
\lim_{\tau\to \infty} \left\|V(y,\tau)- \frac{y}{n} \right\|_{L^2(B_{R_0})} =0\,.
\end{equation}

\end{theorem}
\begin{proof} (i) By displacing the origin of time by one unit, we may assume that our solution satisfies the a priori bound $u(x,t) \le \frac{1}{t+1}$, hence $U(y,\tau)\le 1$. The mass of the renormalized solution will be $M$ at all times $\tau\ge 0$. This time displacement also allows to assume that the initial energy is finite, i.\,e., $\nabla p_0(y)\in L^2(\ren)$.

The boundedness of the orbit $U(\cdot,\tau)$ in $L^1\cap L^\infty$ implies that as $\tau\to\infty$ there is a sequence $\tau_k$ such that $U(\cdot,\tau_k)$ converges weakly to some $U_\infty(y)$, and also $0\le U_\infty(y)\le 1$.

\noindent (ii)
If $n\ge 3$,  we  consider the ``entropy" \begin{equation}
\mbox{Ent}(U):=\frac12\int_{\mr^n}(|\nabla P|^2+ \frac1{n}U\,y^2)\,dy
\end{equation} which is the version for the renormalized flow of the energy \eqref{energys1}.
Recall that $\nabla P=-V$ and $\nabla \cdot V=U$. Not also that under the stated assumptions $\int |\nabla P|^2\,dy=\int UP\,dy$. It is easily calculated that the entropy evolves according to the rule
$$
\frac{d}{d\tau}\mbox{Ent}(U(\tau))=-D(U(\tau)), \qquad
D(U(\tau)):= \int U\,\left|V- \frac{y}{n}\right|^2\,dy\,.
$$
(See  \cite{CV2} for a similar entropy calculation for $s<1$.)

\noindent (ii') If $n=2$, this entropy is not finite and its definition needs to be modified according to Section \ref{n2}: we set $U_0$  to be any smooth compactly supported nonnegative function such that $\int U_0 =\int U$,
and let $P_0= (-\Delta)^{-1}U_0$. We then define
$$\mbox{Ent}(U):=\int_{\mr^2}\(\frac{1}{2} |\nabla (P-P_0) |^2 +U_0(P-P_0) + \frac1{2n}U\,y^2\)\,dy$$ which is finite.
Computing as in Section \ref{n2} we find that
$$\frac{d}{d\tau}\mbox{Ent}(U(\tau))=-D(U(\tau))$$ still holds, and we can continue the proof in the same way as for $n \ge 3$.
\nc

\medskip

\noindent (iii) The fact that $\mbox{Ent}(U(\tau))$ is nonincreasing along the orbit of the renormalized flow implies that there exists the limit
\begin{equation}\label{liment}
E_*(U)=\lim_{\tau\to\infty} \mbox{Ent}(U(\tau))\,.
\end{equation}
Moreover, the integral  $\int_1^\infty\int U\,|V- (y/n)|^2\,dy\,d\tau$ is convergent as an integral in an infinite time interval. Another consequence is that the second moment $\int U(y,\tau)\,y^2\,dy$ is bounded for all $\tau$ (by $2n\,\mbox{Ent}(U(0))$, to be precise), and this control of the behaviour as $y\to \infty$ allows to make sure that in the limit $\int U_\infty(y)\,dy=M.$

\medskip

\noindent (iv) We now take limits as $\tau_k\to\infty$ not only in the family $\{ U(y,\tau_k)\}_k$, $k\to\infty$,  but  in the family of orbits $\{U(y,\tau +\tau_k)\}_k$, for $y\in\ren$ and $0<\tau\le T$ for fixed $T>0$; $k\to\infty$.
The compactness   (which can be simply reproduced from Theorem \ref{thm.exlimit}) allows to show that along a further subsequence that we also call $\tau_k\to\infty$ we get
$$
U(y, \tau+\tau_k)\to \widetilde{U}(y,\tau),
$$
in the weak sense, as well as ${V}(y, \tau+ \tau_k)\to \widetilde{V}(y,\tau)$
 strongly. Notice that  $\widetilde{U}(y,0)  = U_\infty(y)$. Since \eqref{liment} holds,
we also have
 $$
\mbox {Ent}(\widetilde{U}(\tau))=E_*(U), \quad \mbox{a constant; \ and} \ D(\widetilde{U}(\tau))=0\,.
 $$
 The last expression means that for a.e $\tau>0$ and a.e. $y$ we have
 $$\widetilde{U} (\widetilde{V}-(y/n))=0,$$
 which according to the renormalized evolution equation \eqref{ren.eq} and $V= - \nab P$ means that $\partial_\tau \widetilde{U}=0$. \normalcolor Therefore, the limit orbit is a stationary solution of the renormalized equation.

 \medskip

 \noindent (v) We claim that our candidate $U_*(y)$ of \eqref{ustar} is the only stationary solution satisfying the conditions $\int \widetilde{U}(y)\,dy=M$ and $0\le \widetilde{U}\le 1$ and $ \widetilde{V}(y)=y/n$ on the support of $\widetilde{U}$.

 First of all the a priori bound for $\widetilde{V}$ in $L^\infty$ (due to integral kernel estimates on the representation formula performed in Formula \eqref{bornevv})  implies
 that $\widetilde{V}=y/n$ cannot be true for large $y$, therefore $\widetilde{U}$ must be compactly supported. Once we know the compact support property, it is already proven in  \cite{BLL} that the asymptotic limit is $\widetilde{U}= U_*$ (this also follows easily from our analysis of Theorem \ref{thm.7.1}). 
 
 \medskip

 \noindent (vi) The convergence of $U(\cdot,\tau)$ towards $U_*$ is at this point  weak in $L^1$.  But we also have $\sup_y U(y,\tau)\le 1$, this unilateral bound allows to improve the weak convergence to strong convergence in the ball $B_{R_0}(0)$, i.\,e., $\|U(y,\tau)-U_*(y)\|_1\to 0$ in that ball.  But the mass of $U_*$ contained in that ball is the whole mass, so that $U(y,\tau)$ must converge to zero in the complement of the ball. This proves the main asymptotic formula  \eqref{as.from.1}.
  Since the solutions are uniformly bounded, the convergence in all the $L^p$ spaces also follows for $1\le p<\infty$. However, the convergence in $L^\infty$ is impossible due to the discontinuous form of the limit function.

 \medskip

 \noindent (vii) The convergence of the velocities now follows from properties of the convolution formula for $V$ in terms of $U$.

 \medskip

 \noindent (viii) We still have to check that in displacing the origin of times from $t=0$ to $t=1$ we have not lost the property of finite second moment. Indeed, using the equation and several integrations by parts, we have
 \begin{eqnarray*}
& \dfrac{d}{dt} \int u\,x^2\,dx= \int \nabla\cdot(u\nabla p) \, x^2 dx= -2\int x \cdot (u\nabla p) \,dx=
 2\dint  x \cdot \nab p \Delta p  \,dx \\
& = -2 \int \nab p \cdot \nab ( x \cdot \nab p) =  - 2 \int |\nab p|^2 - \int x \cdot \nab |\nab p|^2
 \\
& =(n-2)\dint |\nab p|^2\,dx= (n-2)\int u\,p\,dx.
 \end{eqnarray*}

If $n=2$ then we are done. If $n \ge 3$, we have found that  the time derivative of the second moment is a constant times the energy \eqref{energys1} which is itself nonincreasing in time, hence bounded independently of time. It follows that   the second moment remains bounded for all times. \end{proof}

\subsection{ Asymptotic behaviour for radial solutions}\label{ss.abrad}

We will perform next the proof that the large-time behaviour of finite-mass solutions is given by the self-similar vortex patch solutions of Section \ref{secselfsim} in the case of radially symmetric data, where the Burgers' reformulation allows to use familiar tools. We do not need any further restriction on the class of initial data. A  reference for asymptotic convergence of conservation laws is \cite{LP84}. Section 3 of that paper applies directly in the case $n=1$ and
it also does when $n\ge 2$ after replacing the space variable $r$ by $s$. Since we also want to estimate  $u(r,t)$, which is a weighted  derivative of the mass function $M(r,t)$, we give the details of the argument.

\begin{theorem}\label{th6.1} For any radial solution to \eqref{eqs1} with integrable initial data,  letting $M(x,t)$ be its ``mass function" as in \eqref{M}  and assuming  $\max_{r\ge 0}M(r,0)= M_0$, we have $0\le M(rt^{1/n},t)\le r^n/n$ and
\begin{equation}\label{as.m}
| r^n/n -M(r t^{1/n},t)|\to 0
\end{equation}
uniformly in $\in \mr$ and $t\to\infty$. Moreover, if $u_0$ is compactly supported, the set of $x$ such that $M(|x|,t)<M_0$, i.e. the support of $u(x,t)$, is included in a disk centered at the origin of radius  $S(t)$ such that $S(t)\ge (nM_0t)^{1/n}$ and
\begin{equation}\label{as.i}
S(t)- (nM_0t)^{1/n}\to 0 \quad \text{as } \ t\to\infty.
\end{equation}
\end{theorem}

\begin{proof} (i) We first check the upper bound.  First, since $M(r,t)$ is increasing in $r$,  it has a limit $\lim_{r\to \infty}=\sup_{r} M(r,t)$ which is also equal to $\int u(r,t)\,r^{n-1}dr=M_0$ hence is constant in time.

 In view of the universal estimate of Section \ref{sec.ub} we have $u(x,t)\le 1/t$ so that $M(r,t)\le r^n/nt$ and since $\lim_{r\to \infty} M(r,t)= M_0$, we deduce in view of the above that  $S(t)\ge (nM_0t)^{1/n}$ (this includes that possibility that $S(t)\infty$.  This settles half of the estimates.

(ii) In order to bound the solution from below we consider a vertically displaced function
$$
M_1(r,t)=M(r,t)+\ve, \quad \ve>0.
$$
Then,
$$
(M_1)_t=-\frac1{r^{n-1}}(M_1-\ve) (M_1)_r \ge -\frac1{r^{n-1}} M_1 (M_1)_r
$$
i.e.  $M_1$ is a supersolution of the original Burgers' equation. If $R_0$ be determined by the condition $M(R_0,0)\ge M_0-\ve$. $R_0$ depends on $\ve$ and may be large but is is finite. Then
$$
 M_1(r,0)\ge \ve \quad \text{ for } \ 0\le r\le R_0, \qquad
M_1(r,0)= M_0+\ve \quad \text{ for } \  r\ge R_0
$$

We now consider the vortex patch solution $u_2(r,t) $ with initial  data $\frac{1}{\tau} \chi_{B_{R}}$  and total mass $M_0$. For that we must have $R^n/(n\tau)= M_0$. We  let $M_2(r,t)$ be its mass function.  Then
$$
M_2(r,t)=   \frac{r^n}{n(t+\tau)}  \chi_{\{r \le R ( 1+ t/\tau)^{1/n}\}}\,.
$$
If we choose $\tau>0$ so that $\ve\ge R_0^n/ n\tau \,$, then we will have \ $M_1(r,0)\ge M_2(r,0)$. By the maximum principle  for \eqref{burger}, we deduce
$M_1(r,t)\ge M_2(r,t)$ for all $r\ge 0$ and all $t>0$.
This means that
$$
M(rt^{1/n},t)\ge \frac{r^n}{n(1+(\tau/t))}-\ve \qquad \text{for} \  r\le [n(M_0+\ve)(1+(\tau/t))]^{1/n}\,.
$$
 Passing  to the limit $t\to\infty$, the term in  $\tau$ disappears, and letting then $\ep \to 0$
 we easily arrive at both  estimates \eqref{as.m}.

 (iii) The estimates on the support for compactly supported data are derived in Section 3 of \cite{BLL}, even with rates of convergence.
\end{proof}

As a consequence of this result, we have estimates for the radial velocity $v=M/r^{n-1}$ that we resemble the velocity of the elementary vortex patch $V(r,t)=r/nt$. They are as follows:
the velocity satisfies the estimates
\begin{equation}
v(r,t)\le \frac{r}{nt} \qquad \text{for } 0<r<R(t)
$$
and
$$
\qquad \frac{r}{n}-t^{(n-1)/n} v(rt^{1/n},t)\to 0
\end{equation}
uniformly in $r\le S(t)/t^{1/n}$ as $t\to\infty$. We can also give estimates for the density.

\begin{corollary} The density satisfies the estimates
\begin{equation}
u(r,t)\le \frac{1}{t} \qquad \text{for } 0<r<R(t)
\end{equation}
 and \ $t\, u(rt^{1/n},t)\to \chi_{B_{R_1}(0)}(r)$ in $L^\infty$  (with weak-* convergence) as $t\to\infty$, where $R_1=(nM_0)^{1/n}$.
\end{corollary}
\begin{proof}
It  is convenient to introduce the rescaled density
$ {\hat u}(r,\tau)= t\,u(rt^{1/n},t). $
We know that  $\hat u$ is bounded uniformly for large times by Theorem \ref{th51} and its support is uniformly bounded, by Theorem \ref{th6.1}. It therefore converges in the weak-$*$  topology of $L^\infty(\ren)$ to a function $U$ with compact support. Due to the relation between $u$ and $M$, the limit $U$ must be related to the rescaled mass function of the elementary vortex patch by $r^{n-1}U(r)=tM'(rt^{1/n})=M'(r)$. We conclude that $U(r)$ is the characteristic function of the ball of radius  $R_1$.  Since the limit is unique the convergence does not depend on the subsequence $t_k\to\infty$.\end{proof}

\begin{remark} The example of the  radial solution with an expanding ball together with an expanding annulus shows that the convergence of $\hat u$ toward the equilibrium state $\hat U(r) = \chi_{B_{M_0/n}(0)}(r)$ cannot happen in the uniform topology.\end{remark}



















\vskip 1cm

\bibliographystyle{amsplain}

\

\noindent{\bf Addresses}

\noindent{\sc Sylvia Serfaty:}\\
  UPMC Univ Paris 06, UMR  7598 Laboratoire Jacques-Louis Lions,\\
   Paris, F-75005  France ;\\
   CNRS, UMR 7598 LJLL, Paris, F-75005 France\\
   \&  Courant Institute, New York University,
251 Mercer st, NY NY 10012, USA\\
(e-mail: serfaty@ann.jussieu.fr)

\medskip

\noindent{\sc J.~L. V\'{a}zquez: } \\
Departamento de Matem\'{a}ticas, Universidad Aut\'{o}noma de Madrid, \\
28049 Madrid, Spain. \\
(e-mail: juanluis.vazquez@uam.es). 

\newpage




\end{document}